\documentclass{amsart}

\usepackage{amsmath,amsxtra,amssymb,amsfonts,latexsym, mathrsfs, dsfont,bm,mathtools}
\usepackage{latexsym}
\usepackage{bbm}
\usepackage[usenames]{color}
\usepackage[hyperindex]{hyperref}
\usepackage[initials]{amsrefs}
\usepackage{tikz}
\usepackage{comment}
\usepackage{color}
\usepackage{graphics,epsf}         
\usepackage{enumerate}
\usepackage{color}

\newtheorem{theorem}{Theorem}[section]
\newtheorem{lemma}[theorem]{Lemma}

\newtheorem{remark}{Remark}[section]
\newtheorem{example}{Example}

\newtheorem{definition}{Definition}[section]
\newtheorem{corollary}[theorem]{Corollary}

\newtheorem{claim}{Claim}

\numberwithin{equation}{section}

\renewcommand{\q}{\quad}

\newcommand{\f}{\frac}

\newcommand{\p}{\partial}

\newcommand{\ZR}{\mathbb{R}}

\newcommand{\Id}{{\bf 1}}
\newcommand{\set}[2]{\left\{ #1 \vphantom{#2} \right. : \left. \vphantom{#1} #2 \right\}}

\def\supp{{\text{\rm supp }}}

\makeatletter
\newcommand*{\rom}[1]{\expandafter\@slowromancap\romannumeral #1@}
\makeatother
\newcommand{\bsb}{\boldsymbol}

\begin{document}

\title[Trilinear Estimates]{Maximal Decay Inequalities for Trilinear oscillatory integrals of convolution type}
\author{Philip T. Gressman}
\address{ Department of Mathematics, University of Pennsylvania, Philadelphia, PA 19104, USA}
\email{gressman@math.upenn.edu }

\author{Lechao Xiao}
\address{ Department of Mathematics, University of Pennsylvania, Philadelphia, PA 19104, USA}
\email{xle@math.upenn.edu}
\thanks{The first author is partially supported by NSF grant DMS-1361697 and an Alfred P. Sloan Research Fellowship.}
\date{\today}


\begin{abstract}
In this paper we prove sharp $L^\infty$-$L^\infty$-$L^\infty$ decay for certain trilinear oscillatory integral forms of convolution type on $\ZR^2$. These estimates imply earlier $L^2$-$L^2$-$L^2$ results obtained by the second author as well as corresponding sharp, stable sublevel set estimates of the form studied by Christ \cite{CHR11-1} and Christ, Li, Tao, and Thiele \cite{CLTT05}. New connections to the multilinear results of Phong, Stein, and Sturm \cite{PSS01} are also considered.
\end{abstract}

\maketitle

%
%
%
%
%
%
%
%
%
%
%
%
%
%
%
%
%
%
%
%

\section{Introduction }

Beginning with the groundbreaking work of Christ, Li, Tao, and Thiele \cite{CLTT05}, there has been significant interest in the harmonic analysis literature to develop a robust and general theory of multilinear oscillatory integral operators. However, despite the fundamental insights provided by \cite{CLTT05}, progress on this program has been slow, due to the sheer complexity of the problem and the apparent inadequacy of existing tools in this more general context. Some of the more successful strategies to date focus on special cases of their general framework; see
\cites{CHR11-1, CHR11-2, CS11, GR08, GR11, X2013} for recent progress and \cites{CCW99, PSS01, CW02} for other related topics. 
One such special case, which we will further examine in this paper, is the case of trilinear forms which have convolution-type structure. Specifically we will consider forms $\Lambda(f,g,h)$ given by
\begin{align}\label{TR01}
\Lambda (f,g,h)= \iint e^{i\lambda S(x,y)}f(x)g(y)h(x+y)\phi(x,y) dxdy,
\end{align}
where $S(x,y)$ is a real analytic function defined in a neighborhood of the origin and $\phi$ is 
smooth cut-off function supported sufficiently near the origin. In the present case, we will take $f,g,h$ to belong to $L^p(\ZR)$, $L^q(\ZR)$, and $L^r(\ZR)$, respectively, and study the norm of the form as a multilinear functional on $L^p \times L^q \times L^r$ as $\lambda \rightarrow \pm \infty$. As observed in \cite{CLTT05}, One expects no decay at all when $S(x,y)$ may be written as a sum $S_1(x) + S_2(y) + S_3(x+y)$ for measurable functions $S_1, S_2$, and $S_3$. When $S$ is smooth, the (non-)degeneracy of $S$ is captured by the action of the differential operator $D$
\begin{align*}
D =\partial_x \partial_y (\partial_x -\partial_y)
\end{align*}
which annihilates sums of the form $S_1(x) + S_2(y) + S_3(x+y)$.

In his thesis, the second author considered a special case, namely $f$, $g$ and $h$ all belonging to $L^2$, and showed that
 sharp decay estimates can be characterized by the relative multiplicity of $S$, which is defined to be the multiplicity of the quotient of $S$ by the class of functions annihilated by the differential operator $D$. More precisely, if $n\in\mathbb N$ denotes the relative multiplicity of $S$, then 
 \begin{align}\label{L2}
\left | \Lambda(f,g,h) \right| \lesssim |\lambda|^{-\f 1 {2n}}\|f\|_2\|g\|_2\|h\|_2. 
\end{align}
 The basic observation upon which these earlier results is based is an estimate of the form
\begin{equation} | \Lambda(f,g,h)| \leq C |\lambda|^{-\frac{1}{6}} \|f\|_2 \|g\|_2 \|h\|_2 \label{onesix} \end{equation}
for phases $S$ with $|DS| > c > 0$ on the support of $\phi$, 
which itself follows from a clever application (see \cite{LI08}) of the H\"{o}rmander argument for nondegenerate bilinear oscillatory integral forms. 
Despite the fact that scaling arguments show that the exponent of $\lambda$ in \eqref{onesix} cannot be improved, a comparison to the sublevel set estimate (namely, the decay rate of the volume of the set
\[ \set{(x,y) \in [0,1]^2}{ |S(x,y)| < \epsilon} \]
as $\epsilon \rightarrow 0^+$) suggests that the decay rate $|\lambda|^{-1/6}$ is likely not the best possible if one considers $L^p$ spaces on the right-hand side other than $L^2$. Our first result is that such an improvement over $\lambda^{-1/6}$ can indeed be achieved:
\begin{theorem}\label{intro}
Let $S(x,y)$ be a real analytic function defined in a neighborhood of $0$. Assume that 
\begin{align*}
|\partial_x\partial_y(\partial_x -\partial_y)S(x,y)| \geq 1,
\end{align*}
for all $(x,y) \in {\rm Conv}(\supp(\phi))$. 
Then 
\begin{align*}
\left | \Lambda(f,g,h) \right| \lesssim |\lambda|^{-1/4}\|f\|_p\|g\|_q\|h\|_r,
\end{align*}
for all $(p,q, r)\in\mathcal A$.
\begin{align}
\mathcal A =\left\{(p,q,r): p,q, r\in [2,4), ~ \frac{1} p + \frac 1 q +\frac 1 r = \frac 5 {4} \right\}. 
\end{align}
\end{theorem}
We note that this improvement from $|\lambda|^{-1/6}$ to $|\lambda|^{-1/4}$ still falls short of the $|\lambda|^{-1/3}$ decay rate suggested by the optimal sublevel set decay. 
 Our investigation of this final gap suggests that improvements beyond $|\lambda|^{-1/4}$, if possible, are likely extremely difficult to accomplish.

When $S$ is sufficiently degenerate, however, closing the gap between the sublevel set decay rate and the oscillatory integral decay rate is, in fact, possible. While this matching of decay rates is satisfying and natural, it should perhaps be regarded as somewhat surprising that this is possible since, among other things, the highest-possible decay rate can 
 only be achieved when $f,g$, and $h$ all belong to $L^\infty$, which is not traditionally a regime in which one expects to find strong cancellation effects.

To state our main results for general analytic phases, we introduce the following notions. 
Let $P =DS$ and write 
\begin{align}\label{DS01}
P(x,y) =\sum_{j=n-3}^\infty P_j(x,y),
\end{align}
where each $P_j$ is a homogeneous polynomial of degree $j$. 
Here we implicitly assumed the relative multiplicity of $S$ is equal to $n$. 
Let $\alpha, \beta, \gamma$ be the orders of $x$, $y$, $(x+y)$ in $P_{n-3}$, $d_0$ be the maximal order of other linear terms in $P_{n-3}$. 
We must also define $d_1$ to be the maximal vanishing order of any edge whose slope is not equal to $-1$ of the Newton polygon of $P$ in any of the coordinates $(x,y)$, 
or $(x, x+y)$ or $(y, x+y)$ 
(if the reader is unfamiliar with the Newton polygon, the notion of order of vanishing of an edge is made precise by Definition \ref{VO1}).
It is worth pointing out that $d_1 \leq \max\{\alpha,\beta, \gamma\}$.  
With these definitions in place, we define
\begin{align}\label{AS}
\kappa  = \max\{\alpha, \beta, \gamma, d_0+1, d_1+1\},
\end{align} 
and come to our main result on the decay of $\Lambda(f,g,h)$ on $L^\infty \times L^\infty \times L^\infty$:
\begin{theorem}\label{main3}
Let $S(x,y)$ be as above. Then 
\begin{align}\label{MI01}
|\Lambda(f,g,h)|\leq C|\lambda|^{-\frac {1}{\max\{4, \kappa +2, \frac n 2\}}} \log(2+|\lambda|)^\mu\|f\|_\infty\|g\|_\infty\|h\|_\infty,
\end{align}
where the exponent $\mu$ satisfies
\begin{enumerate}
\item (Sharp Estimates) when $n\geq 9$, 
\begin{enumerate}
\item $\mu = 0$ if $\kappa < \frac n 2-2$ or  if $d_1+1=\kappa =\f n 2 -2$ and $\alpha, \beta, \gamma, d_0+1< \kappa$,
\item $\mu = 1$ if $\f n 2 -2 =\kappa$ and at least one of $\alpha, \beta, \gamma, d_0+1$ is equal to $\kappa$. 
\end{enumerate}
\item (Non-Sharp Estimates) 
\begin{enumerate}
\item $\mu =0$ if $n\geq 9$ and $\kappa >\f n 2 -2$ or if $n=3$,
\item $\mu = 1$ if $4 \leq n \leq 7$,
\item  $\mu=2$ if $n =8$. 
\end{enumerate}
\end{enumerate}
\end{theorem}
For $n\geq 9$ and 
for a generic function $S$ (i.e. $\kappa <\frac n 2 -2$),
the estimates in Theorem \ref{main3} strictly improve upon (\ref{L2}) in the sense that (\ref{L2}) can be obtained
by interpolating (\ref{MI01}) with the trivial non-decay $L^{\frac 3 2}\times L^{\frac 3 2}\times L^{\frac 3 2}$ estimate.
Moreover, the decay in (\ref{MI01}) also exceeds the decay rates found by Phong, Stein, and Sturm \cite{PSS01} for the same phase function and a corresponding {\it bilinear} form. 
For a generic homogeneous polynomial in $\mathbb R^2$ of degree $n \geq 9$, 
the highest decay achieved for any estimate in \cite{PSS01} is $|\lambda|^{-1/n}$, whereas
 Theorem \ref{main3} yields a decay of $|\lambda|^{-2/n}$.

Although \eqref{MI01} does not have a sharp rate of decay in general when $\kappa >\f n 2 -2$,
it is possible to obtain better and even sharp estimates for certain phases in this category. We will not pursue this issue in the present paper. 
On the other hand, improvement of the decay when $3\leq n \leq 8$ and $\kappa \leq \f n 2 -2$ seems to be an extremely challenging problem, for it 
would require improvement of the decay in Theorem \ref{intro}. In general, a natural goal for this problem is to fully understand the maximal convex hull
of the tuples $(\delta, \f 1 p, \f 1 q, \f 1 r)$ such that 
\begin{align}
|\Lambda (f,g,h)|\lesssim |\lambda|^{-\delta}\|f\|_p\|g\|_q\|h\|_r;
\end{align}
Our preliminary investigation of this question suggests the answer will be extremely complicated, in contrast, for example to the simply-stated results of Phong, Stein, and Sturm \cite{PSS01}. 
Even assuming the best possible estimate for the non-degenerate case 
\begin{align}
|\Lambda (f,g,h)|\lesssim |\lambda|^{-\f 1 3}\|f\|_3\|g\|_3\|h\|_3,
\end{align}
it is still unclear to us how many extreme points the convex polytope will have for general analytic phases, not to mention the question of where those extreme points should be located.

The proof of Theorem \ref{main3} is built up upon two main ingredients. 
The first is a localized version of Theorem \ref{intro} that will be developed 
in Section \ref{SE02}.  The proofs of the local and global versions of Theorem \ref{intro} are essentially the same. 
The second main ingredient is a resolution of singularities algorithm developed in \cite{X2013}. 
We only sketch its proof in Section \ref{PRE1} but refer the reader to \cite{X2013} for rigorous details. 
The proof of the main theorem will appear in the last section. The basic framework is as follows: 
 we first employ the resolution algorithm to decompose
the support of $\phi$ into finitely many subregions on which $DS$ behaves like a monomial. 
One can then attack the corresponding localizations (\ref{TR01})
by invoking Theorem \ref{local}. 

\subsection*{Notation:} 
We use $X\lesssim Y$ to mean ``there exists a constant $C>0$ such that $X\leq CY $,'' where in this context (and throughout the entire paper) constants $C$ may depend on the phase $S$ and the cutoff function $\phi$ but must be independent of the functions $f$, $g$, and $h$ and the real parameter $\lambda$. The expression $X\gtrsim Y$ is analogous,
and $X\sim Y$ means both $X\lesssim Y$ and $X\gtrsim Y$.

%
%
%
%
%
%
%
%
%
%
%
%
%
%
%
%
%
%
%
%

\section{The case of non-degeneracy and its local analogue}
\label{SE02}

The goal of this section is to prove Theorem \ref{intro} (referred as the global version) and develop its local analogue.  
The proof will be accomplished in two steps. In the first step, the method of $TT^*$ is employed to 
reduce the trilinear form to a bilinear variant to which one may apply either a 
classical result from H\"ormander (in the proof of the global version) or the Phong-Stein operator van
der Corput Lemma (in the local version). This basic strategy was introduced by Li \cite{LI08} in his proof of the $|\lambda|^{-\frac 1 6}$ decay rate for $L^2 \times L^2 \times L^2$.
To improve the decay to $|\lambda|^{-\frac 1 4}$, the second step of our proof invokes the 
Hardy-Littlewood-Sobolev inequality to take advantage of the convolutional structure of our form and improve the overall decay in $\lambda$.    
We begin with the following definition.
\begin{definition}
Let $\vec v$ be a vector in $\mathbb R^2$.
A set $X\subset \mathbb R^2$ is $\vec v$-convex when for any points $a,b \in X$ such that $a - b = c \vec v$ for some real $c$, the line segment joining $a$ and $b$ is also contained in $X$.
We use ${\rm Conv}_{\vec v}(X)$ to denote the smallest $\vec v$-convex set containing $X$
and ${\rm Conv}_{V}(X)$ the smallest set containing $X$ who is $\vec v$-convex for all $\vec v\in V$.  
\end{definition}
\begin{theorem}\label{local}
Let  $\phi(x,y)$ be a smooth function supported in a strip of $x$-width and $y$-width
no more than $\delta_1$ and $\delta_2$ respectively. Assume also that
\begin{align}\label{ps1}
|\partial_y \phi(x,y)| \lesssim \delta_2^{-1} \q\q\mbox{and} \q\q|\partial^2_y \phi(x,y)| \lesssim \delta_2^{-2}. 
\end{align}
Set $\vec u = (1,0)$, $\vec v =(0,1)$ and $\vec w = (-1,1)$, and suppose for some $\mu>0$, the amplitude $S(x,y)$ is a smooth function such that for all $(x,y)\in {\rm Conv}_{\{\vec u, \vec v, \vec w\}}(\supp(\phi))$:
\begin{align}\label{LE01}
|DS(x,y)| \gtrsim \mu
\q\q\mbox{and}\q\q
 |\partial_{y}^{l}DS(x,y)| \lesssim \frac{\mu}{\delta_2^{l}} \q\mbox{for}\q l=1,2.
\end{align}
Then for $\Lambda$ defined as above, one has
\begin{align}\label{LO01}
\left | \Lambda(f,g,h) \right| \lesssim |\lambda\mu|^{-1/4}\|f\|_p\|g\|_q\|h\|_2,
\end{align}
and 
\begin{align}\label{LO02}
\left | \Lambda(f,g,h) \right| \lesssim |\lambda\mu|^{-1/4}\|f\|_2\|g\|_p\|h\|_q,
\end{align}
for $2<p, q<4$ and $\f 1 p  +\f 1 q = \f 3 4$. 
\end{theorem}
\begin{remark}
If the estimates (\ref{ps1}) and (\ref{LE01}) are replaced by
\begin{align}\label{ps2}
|\partial_x \phi(x,y)| \lesssim \delta_1^{-1}, \q\q|\partial^2_x \phi(x,y)| \lesssim \delta_1^{-2}. 
\end{align}
and 
\begin{align}\label{LE02}
|DS(x,y)| \gtrsim \mu
,\q\q
 |\partial_{x}^{l}DS(x,y)| \lesssim \frac{\mu}{\delta_1^{l}} \q\mbox{for}\q l=1,2,
\end{align}
then the corresponding estimates
\begin{align}\label{LO03}
\left | \Lambda(f,g,h) \right| \lesssim |\lambda|^{-1/4}\|f\|_p\|g\|_q\|h\|_2,
\end{align}
and 
\begin{align}\label{LO04}
\left | \Lambda(f,g,h) \right| \lesssim |\lambda|^{-1/4}\|f\|_p\|g\|_2\|h\|_q
\end{align}
also hold. One can then employ interpolation to obtain the following corollary.
\end{remark}

\begin{corollary}\label{LC90}
Let $\phi(x,y)$ be a smooth function supported in a strip of $x$-width and $y$-width
no more than $\delta_1$ and $\delta_2$ respectively. Suppose that for $l = 1 $ and $2$, 
\begin{align}\label{ps10}
|\partial_y \phi(x,y)| \lesssim \delta_2^{-l} \q\q\mbox{and} \q\q|\partial^l_x \phi(x,y)| \lesssim \delta_1^{-l}. 
\end{align}
Let $\mu>0$ and $S(x,y)$ be a smooth function s.t. for all $(x,y)\in {\rm Conv}_{\{\vec u, \vec v, \vec w\}}(\supp(\phi))$ and for $l=1$ and $2$, 
\begin{align}\label{LE010}
|DS(x,y)| \gtrsim \mu, 
\q
 |\partial_{x}^{l}DS(x,y)| \lesssim \frac{\mu}{\delta_1^{l}},  \q
\mbox{and}\q
 |\partial_{y}^{l}DS(x,y)| \lesssim \frac{\mu}{\delta_2^{l}}.
\end{align}
Then for $\Lambda$ defined as above and for $(p,q,r) \in \mathcal A$, one has
\begin{align}\label{LO010}
\left | \Lambda(f,g,h) \right| \lesssim |\lambda\mu|^{-1/4}\|f\|_p\|g\|_q\|h\|_r. 
\end{align}
\end{corollary}

\begin{proof}[Proofs of Theorem \ref{intro} and Theorem \ref{local}]
First, we use duality to eliminate dependence on the function $h$. Under the change of variables $u=x+y$ and $v=y$, we must have
\begin{align}\label{pullout}
|\Lambda(f,g,h)| \leq \|B(f,g)\|_2\|h\|_2,
\end{align}
where 
\begin{align*}
B(f,g)(u) = \int e^{i\lambda S(u-v,v)} f(u-v) g(v) \phi (u-v, v) dv.
\end{align*}
Employing the method of $TT^*$ yields that $\|B(f,g)\|_2^2$ equals
\begin{align*}
  \iiint e^{i(\lambda S(u-v_1,v_1)-\lambda S(u-v_2,v_2))}  f(u-v_1) &  \bar f(u-v_2)  g(v_1) \bar g(v_2)
\\
&  \times \phi(u-v_1, v_1) \phi (u-v_2, v_2) dv_1dv_2du.
\end{align*}
Change variables again: let $x=u-v_1$, $y=v_1$ and $\tau = v_2-v_1$ and set
\begin{align}
&S_\tau(x,y)  = S(x,y) -S(x-\tau,y+\tau), \nonumber
\\
&F_\tau(x) = f(x)\bar f(x-\tau), \nonumber
\\
&G_\tau(y) =g(y)\bar  g(y+\tau), \nonumber
\\
&\phi_\tau(x,y) = \phi(x,y)\phi(x-\tau,y+\tau).  \label{support}
\end{align}
With this notation, we have the identity
\begin{align}\label{bound}
\|B(f,g)\|_2 ^2 = \int\left(\iint e^{i\lambda S_\tau(x,y)} F_\tau(x) G_\tau(y) \phi_\tau(x,y) dxdy \right) d\tau.
\end{align}
We claim that under the assumptions of Theorem \ref{intro} 
\begin{align}\label{HO1002}
&\|B(f,g)\|_2 ^2
\leq C\int |\lambda \tau|^{-1/2} \|F_\tau\|_2\|G_\tau\|_2 d\tau
\end{align}
by the H\"{o}rmander theory of bilinear oscillatory integrals and that under the assumptions of Theorem \ref{local},
\begin{align}\label{HO1003}
&\|B(f,g)\|_2 ^2
\leq C\int |\lambda\mu \tau|^{-1/2} \|F_\tau\|_2\|G_\tau\|_2 d\tau 
\end{align}
by virtue of the operator van der Corput estimate of Phong and Stein. We will postpone the derivation of these inequalities for the moment and first show how to deduce the results of Theorem \ref{intro} and Theorem \ref{local}. In doing so, we 
focus only on the proof of Thoeorem \ref{intro} since the deduction of Theorem \ref{local} at this point differs only by a constant factor of $\mu^{-1/2}$.

By applying the Cauchy-Schwarz inequality, $\|B(f,g)\| ^2$ is then controlled by 
\begin{align}
 C|\lambda|^{-1/2} \left(\int |\tau|^{-\sigma}\|F_\tau\|^2_2d\tau\right)^{1/2}\left (\int |\tau|^{-\rho}\|G_\tau\|^2_2 d\tau \right)^{1/2},
\end{align}
where $\sigma, \rho > 0$ with $\sigma +\rho=1$. 
Let 
$$
I_\sigma(f)(x)  =\int |\tau|^{-\sigma} f(x-\tau)d\tau. 
$$
Then 
\begin{align*}
\int |\tau|^{-\sigma}\|F_\tau\|^2_2d\tau &= \int |f|^2(x) \Big( \int   |\tau|^{-\sigma} \cdot  |f|^2(x-\tau)d\tau\Big)dx
\\
&= \int |f|^2(x) I_\sigma (|f|^2)(x)dx 
\end{align*}
and
\begin{align*}
\int |\tau|^{-\rho}\|G_\tau\|^2_2 d\tau &= \int |g^2(y)| \Big( \int   |\tau|^{-\rho} \cdot  |g|^2(y+\tau)d\tau\Big)dy
\\
&= \int |g|^2(y) I_\rho (|g|^2)(y)dy .
\end{align*}
The H\"older inequality and the Hardy-Littlewood-Sobolev inequality  yield
\begin{align*}
\int |f|^2(x) I_\sigma (|f|^2)(x)dx
\leq  \big\| |f|^2\big\|_{q'}\big\|I_\sigma (|f|^2)\big\|_q
\leq \| |f|^2\|_{q'}\||f|^2\|_p,
\end{align*}
given 
$$
\frac{1}{p} = 1-\sigma +\frac 1 q    \quad \mbox{and} \quad 0<\sigma <1.
$$
Set $p =q'$, then  
$$
\frac 1 {q'} = \frac{1}{p} = 1-\sigma +\frac 1 q,     
$$
i.e. $\frac 1 {q'} = 1 -\frac \sigma 2$. Thus 
\begin{align*}
\int |f|^2(x) I_\sigma (|f|^2)(x)dx \lesssim \||f|^2\|_{\frac{1}{1- \sigma/ 2}}^2
\end{align*}
and similarly 
\begin{align*}
\int |g|^2(y) I_\rho (|g|^2)(y)dy \lesssim \||g|^2\|_{\frac{1}{1- \rho/ 2}}^2.
\end{align*} 
Therefore 
\begin{align*}
|\Lambda (f,g,h)|\leq \|B(f,g)\|_2\|h\|_2 \lesssim \lambda^{-1/4} \|f\|_p\|g\|_q \|h\|_2
\end{align*}
for 
$$
2<p, q<4 \quad \mbox{and}\quad \frac 1 p +\frac 1 q  = \frac 1 2 (1-\frac \sigma 2)+\frac 1 2 (1-\frac \rho 2) =\frac 3 4. 
$$
This gives (\ref{LO01}). 
Notice that,  in (\ref{pullout}) if we pull out $f$ rather than $h$, then same arguments yield (\ref{LO02}).

It remains to verify (\ref{HO1002}) and (\ref{HO1003}). 
Notice  
\begin{align}\label{to88}
\partial_x\partial_yS_\tau(x,y) = -\int_{0}^\tau DS(x-t, y+t)  dt,
\end{align}
if $\tau >0$; otherwise we need to reverse the limits in the integral. For convenience, we aways assume $\tau >0$.   
The assumptions Theorem \ref{intro} that $|DS(x,y)| \geq 1$ for all $(x,y)\in {\rm Conv}(\supp \phi)$ and the Mean Value Theorem imply that
\begin{align}\label{to}
|\partial_x\partial_yS_\tau(x,y)| \geq |\tau|\, .
\end{align}
A classical result from H\"ormander \cite{HOR73} gives 
\begin{align}
&\|B(f,g)\|_2 ^2
\leq C\int |\lambda \tau|^{-1/2} \|F_\tau\|_2\|G_\tau\|_2 d\tau. 
\end{align}
To verify (\ref{HO1003}), we need the operator van der Corput Lemma from Phong and Stein \cite{PS97}; see also \cite{RY01, GR04}.
\begin{lemma}
Let $\phi(x,y)$ be the same as Theorem \ref{local}.
Let $\mu>0$ and $S(x,y)$ be a smooth function s.t. for all $(x,y)\in {\rm Conv}_{\vec v }(\supp(\phi))$:
\begin{align}\label{PS1}
|\partial_x\partial_yS(x,y)| \gtrsim \mu
\q\q\mbox{and}\q\q
 |\partial_x\partial_y^lS(x,y)| \lesssim \frac{\mu}{\delta_2^{l}} \q\mbox{for}\q l=1,2\,.
\end{align}
Then the operator defined by 
\begin{align}
Tf(x)  =\int e^{i\lambda S(x,y)}f(y)\phi(x,y)dy 
\end{align}
satisfying 
\begin{align}\label{PS2}
\|Tf\|_2 \lesssim |\lambda\mu|^{-\frac 1 2 }\|f\|_2.
\end{align}
\end{lemma}
It is clear (\ref{HO1003}) follows by this lemma once we show that $S_\tau$ and $\phi_\tau$ 
satisfy its assumptions. 
The verification for $\phi_\tau$ is quite straightforward. 
Indeed $\supp (\phi_\tau) \subset \supp (\phi)$ and (\ref{ps1}) follows by the product rule of derivatives.   
To verify (\ref{PS1}), we need to utilize the convexity assumption.  
Let $(x,y_1), (x,y_2)\in \supp \phi_\tau$. 
Then the following four points lie in $\supp \phi$: $(x,y_1), (x,y_2), (x-\tau, y_1+\tau), (x-\tau,y_2+\tau)$, which implies their convex
hull also lies in ${\rm Conv}_{\{\vec v, \vec w\}}(\supp \phi)$. If $(x, y_0)$ is any point between $(x, y_1)$ and $(x, y_2)$ and $t$ is any number between 
$0$ and $\tau$, then $(x-t, y_0+t)$ lies in the convex hull of the above four points and thus in ${\rm Conv}_{\{\vec v, \vec w\}}(\supp \phi)$.    
By (\ref{LE01}) and (\ref{to88}), one has  
\begin{align}
|\partial_x\partial_y S_\tau (x,y_0)| = \left| \int_{0}^\tau DS(x-t, y_0+t)dt\right|  \gtrsim \tau \mu,  
\end{align}
where we have used the Mean Value Theorem in the last inequality. This gives the first estimate of (\ref{PS1}); the second one can be obtained in a same way. 
\end{proof}

%
%
%
%
%
%
%
%
%
%
%
%
%
%
%
%
%
%
%
%
%

\section{Resolution of singularities}\label{PRE1}
In the previous section, we developed the main analytic tool for the proof of the main theorem. 
A crucial step in applying this tool is to establish useful lower bounds of $|DS|$, which is the objective of the current section. 
 More precisely, we employ the resolution algorithm developed in \cite{X2013} to decompose 
a neighborhood of the origin (assuming $DS(0,0)=0$) into finitely many regions on which
$DS$ behaves like a monomial. 

To begin, we introduce the following concepts. 
Let $P(x,y)$ be a real analytic function defined on some neighborhood of the origin. In the proof of our main theorem,
we will take $P(x,y) =DS(x,y)$.  
Write the Taylor series expansion of $P$ as 
\begin{align*}P(x,y)  = \sum\limits_{p,q\in \mathbb N}c_{p,q}x^py^q
\end{align*}
and drop all the terms with $c_{p,q} =0$ in the above expression.
The Newton polygon of $P$ is defined to be  
\begin{align}\label{diagram}
\mathcal N(P) = {\rm Conv} \left(\bigcup_{p,q} \{(u,v)\in \mathbb R^2: u\geq p, v\geq q\,\,\mbox{with}\,\, c_{p,q}\neq 0\} \right).
\end{align}
The Newton diagram is the boundary of $\mathcal N(P)$, which consists of two non-compact edges, a finite (and possibly empty) collection of compact edges $\mathcal E(P)$, and a finite collection of vertices $\mathcal V(P)$. The vertices and the edges are called the faces of the Newton polygon and the set of faces is denoted by $\mathcal F(P)$. The vertex that lies on the bisecting line $p=q$, or if no such vertex exists, the edge that intersects $p =q$ is called 
the main face of $\mathcal N(P)$.
For $F\in \mathcal F(P)$, define 
\begin{align}\label{edgeP}
P_F(x,y) = \sum\limits_{(p,q)\in F} c_{p,q} x^p y^q.
\end{align}
The set supporting lines of $\mathcal N(P)$, denoted by $\mathcal {SL}(P)$, are the lines that intersect the boundary of $\mathcal N(P)$ and do not intersect any other points of $\mathcal N(P)$. 
Notice that each supporting line contains at least one vertex of $\mathcal N(P)$ and each edge of $\mathcal N(P)$ lies in exactly one supporting line. Thus, 
we will also identify an edge with the supporting line containing it.  
There is a one-to-one correspondence $\mathcal M$ between $\mathcal {SL}(P)$ and the set of numbers $[0,\infty]$, given by defining
$\mathcal M(L)$ to be the negative reciprocal of the slope of $L$ for each $L\in \mathcal {SL}(P)$. 
We will often use the notation $L_m\in \mathcal {SL}(P)$ to refer the supporting line with $\mathcal M (L_m)= m$.

\begin{theorem}\label{rs90}
For each analytic function $P(x,y)$ defined in a neighborhood of $0\in\mathbb R^2$, 
there is an $\epsilon>0$ such that, up to a measure zero set, one can 
partition $U= (0,\epsilon)\times (-\epsilon,\epsilon)$ into a finite collection of `curved triangular' regions 
$\{
U_{n,g, {\bsb \alpha}, j}
\}
$
on each of which $P(x,y)$ behaves like a monomial in the following sense. 
For each $U_{n,g, {\bsb \alpha}, j}$, there is change of variables 
\begin{align}\label{CV01}
\rho_n^{-1}(x_n,y_n)=(x,y) 
\end{align}
given by 
\begin{align}\label{change08}
\begin{cases}
 x_n= x 
 \\ 
 y_n= \gamma_n(x) +y_n   x^{m_0+\dots+m_{n-1}}  \end{cases}
\end{align}
with $\gamma_n(x)$ either given by a convergent fractional power series  
 \begin{align}\label{gamma1x}
 \gamma_n(x) = \sum\limits_{k=0}^{\infty} r_kx^{m_{0}+m_{1}+\cdots+m_k},
 \end{align}
or  a polynomial of some fractional power of $x$
 \begin{align}\label{gamma2x}
 \gamma_n(x) = \sum\limits_{k=0}^{n-1} r_kx^{m_{0}+m_{1}+\cdots+m_k},
 \end{align}
such that for any pre-seclected $K\in \mathbb N$, for all $0\leq a, b\leq K$ and $(x,y)\in U_{n,g, {\bsb \alpha}, j}$ one has
\begin{align}
&|P_n(x_n,y_n)| \sim |x_n^{p_{n}}y_n^{q_{n}}| \label{key1}
\\
&|\partial_{x_n}^{a} \partial_{y_n}^{b}P_n(x_n,y_n)| \lesssim  \min\{1, | x_n^{p_{n}-a}y_n^{q_{n}-b} \}\label{key2}|
\end{align}
and
\begin{align}\label{new}
| \partial_{y}^{b}P(x,y)|  \lesssim \min\{1, |x^{p_{n} - b(m_0+\dots +m_{n-1})}y_n^{q_{n}-b}|\}.
\end{align}
Here $P_n$ is defined by
$$
P_n(x_n,y_n) = P (\rho_n^{-1} (x_n,y_n)) = P(x,y)
$$
and $(p_n,q_n)$ is some vertex of the Newton polygon of $P_n$. 
\end{theorem}
The meaning of the subindices $(n,g, {\bsb \alpha}, j)$ 
will be clear after the proof of this theorem.  
The reason we phrased this theorem only in the right-half plane is to avoid writing the absolute value of $x$. Indeed, 
by changing $x$ to $-x$, the above theorem also applies to $U=(-\epsilon, 0)\times (-\epsilon, \epsilon)$, consequently to $U= (-\epsilon,\epsilon)\times (-\epsilon,\epsilon)$ and any its open subset. 
See \cite{X2013} for a proof of this theorem; below we provide only a sketch of the theorem for the purpose of outlining some useful ideas and introducing some terminology for later sections. In particular, we will need to further decompose each $U_{n,g, {\bsb \alpha}, j}$ into ``rectangular boxes''.

\subsection{A sketch of the proof}

Let $U$ be as in the theorem above. 
Set 
\begin{align*}
\begin{cases}
U_0 = U,
\\
P_0 = P,
\\
(x_0,y_0) = (x,y),
\end{cases}
\end{align*}
which can viewed as the original input for the algorithm below. 
The subindex 0 is used here to indicate the 0-th stage the iteration. At each stage, iterations are always performed on some triple $[U, P, (x,y)]$. Here $P$ is a convergent (fractional) power series 
 in $(x,y)$, which are local coordinates centered at the origin, and  $U=(0,\epsilon)\times (-\epsilon,\epsilon)$ with $\epsilon>0$ being a small number depending on $P$. For convenience, we refer to such a triple $[U, P, (x,y)]$ as a standard triple. 

Choose one supporting line $L_m\in \mathcal N(P)$ which contains at least one vertex $(p_0,q_0)=V\in\mathcal V(P) $. 
Let $E_l$ and $E_r$ be the two edges on the left and right of $V$, respectively. Set $m_l =\mathcal M(E_l)$ and $m_r =\mathcal M (E_r)$. Then $0\leq m_l\leq m \leq  m_r\leq \infty$. 
We will consider the region $|y|\sim x^m$ in each of the following three cases: 
\textbf{Case (1).} $m_l<m<m_r$, 
\textbf{Case (2).} $m=m_l$ and 
\textbf{Case (3).} $m=m_r$. 
In \textbf{Case (1)}, we have informally that the vertex $V$ `dominates'  $P(x,y)$, while in \textbf{Case (2)} and \textbf{Case (3)} we have that $E_l$ dominates $P(x,y)$ and $E_r$ dominates $P(x,y)$,  respectively.

In \textbf{Case (1)}, $p_0+m{q_0} < p+mq$ for any other $(p,q)$ with $c_{p,q} \neq 0$. 
Thus in the region $|y|\sim x^m$, when $|x|$ it sufficiently small, the monomial 
$$
|P_V(x,y)|=|c_{p_0,q_0}x^{p_0}y^{q_0}|\sim |x^{p_0+mq_0}|
$$ 
is the dominant term in $P(x,y)$, since 
$$
|P(x,y)-P_V(x,y)|=O(x^{p_0+mq_0+\nu}) \q \textrm {for some} \q\nu>0.
$$
 Thus 
\begin{align}
P(x,y)\sim P_V(x,y)=c_{p_0,q_0}x^{p_0}y^{q_0}
\end{align}
and we can make 
$$
\frac {|P(x,y)-P_V(x,y)|}{|P_V(x,y)|}
$$
to be arbitrarily small by choosing $\epsilon $ sufficiently small, i.e., by shrinking the region $U$.  
Moreover, for any pre-selected $a, b\geq 0$, 
\begin{align}\label{DE01}
|\p_x^a\p_y^b P(x,y)|\lesssim \min\{1, |x^{p_0-a} y^{q_0-b}|\}.
\end{align}
We refer such a region as a `good' region defined by the vertex $V$ and denote it by $U_{0,g,V}$. 
\begin{center}
\begin{tikzpicture}[scale=0.8]
\draw [<->,thick] (0,9) node (yaxis) [above] {$y$}
        |- (9,0) node (xaxis) [right] {$x$};
  \draw[help lines] (0,0) grid (8,8);
     \fill[gray!30] (2,8)--(2,4)--(3,2)--(5,1)--(8,1)--(8,8);
       \draw[thick] (5,0)--node[below left] {$l$}(0,5) ;
       \fill (5,1)   circle (2pt) (3,2) circle (2pt) (2,4)  circle (2pt);
        \draw [thick]	(9,1)--
        			(5,1) node[below left] {$V_3$}--
			(3,2)node[below left] {$V_2$}--
			(2,4)node[below left] {$V_1$}--
			(2,9);
	\node at (5,5) {Newton polygon};
	\node at (2.7,4){(2,4)};
	\node at (3.5,2.3) {(3,2)};
	\node at (5,1.5) {(5,1)};
	\node at (11,6) {$P(x,y) = x^5y-x^3y^2+x^2y^4$};
	\node at (10.2,5) {$P_{V_2}(x,y) = -x^3y^2$};
	\node at (11,4) {$|x|^{2}\lesssim |y| \lesssim |x|^{1/2}$};
	\node at (7,9.5) { \emph{Case (1): The vertex} $V_2$\emph{ is dominant, where} $1/2< m < 2$};
	\node at (4.5,-1) {Figure 1.};
\end{tikzpicture}
\end{center}

\begin{center}
\begin{tikzpicture}[scale=0.8]
\draw [<->,thick] (0,9) node (yaxis) [above] {$y$}
        |- (9,0) node (xaxis) [right] {$x$};
  \draw[help lines] (0,0) grid (8,8);
     \fill[gray!30] (2,8)--(2,4)--(3,2)--(5,1)--(8,1)--(8,8);
       \fill (5,1)   circle (2pt) (3,2) circle (2pt) (2,4)  circle (2pt);
        \draw [thick]	(9,1)--
        			(5,1) node[below left] {$V_3$}--
			(3,2)node[below left] {$V_2$}--
			(2,4)node[below left] {$V_1$}--
			(2,9);
	\node at (5,5) {Newton polygon};
	\node at (2.7,4){(2,4)};
	\node at (3.5,2.3) {(3,2)};
	\node at (5,1.5) {(5,1)};
	\node at (11,6) {$P(x,y) = x^5y-x^3y^2+x^2y^4$};
	\node at (11,5) {$P_{V_1V_2}(x,y) = -x^3y^2+x^2y^4$};
	\node at (12,4) {$|y|\sim |x|^{1/2}$};
	\draw[thick] (0,8)--(4,0);
	\node at (6,9.5) { \emph {Case(2): The edge} $V_1V_2$ \emph{ is dominant, where} $ m =1/ 2$};
		\node at (4.5,-1) {Figure 2.};

\end{tikzpicture}
\end{center}

 \textbf{Case (2)} and \textbf{Case (3)} are essentially the same, but are much more complicated than \textbf{Case (1)}.  We focus on \textbf{Case (2)}. 
 Set  $m_0 =m_l$, $P_0 =P$ and $E_0 =E_l$. One can see that $p_0+m_0q_0 =p+m_0q$ for all $(p,q)\in E_0$ and $p_0+m_0q_0 <p+m_0q$
  for all $(p,q) \notin E_0$ and $c_{p,q}\neq 0$. 
 Then for all $(p,q)\in E_0$, 
 $|x^{p}y^q|\sim |x^{p_0}y^{q_0}|$ in the region $|y| \sim |x|^{m_0}$. Set $y  = rx^{m_0}$ and let $\{r_*\} $ be the set of nonzero roots of $ P_{0,E_0}(1,y)$ with orders $\{s_*\}$.  
 \begin{definition}\label{VO1}
 For each compact edge $E_0$ of $\mathcal N(P)$, 
the number $\max \{s_*\}$ is called the vanishing order of $E_0$ in the right half-plane. 
The vanishing order in the left half-plane is defined similarly by considering the orders of all non-zero roots
of $P_{0,E_0}(-1,y)$. Finally, we define the vanishing order of $E_0$ to be the larger one between them.    
\end{definition}
 
 {From the definitions of $\alpha, \beta, \gamma, d_0$ and $d_1$ in the first section, we have $\max\{s_*\}\leq \max\{\gamma, d_0\}$
  if $\mathcal M(E_0) =1$ and  
 $\max\{s_*\}\leq d_1$ if $\mathcal M(E_0) \neq1 $.  Notice that to compute $d_1$, one needs to compute the maximum $\max\{s_*\}$ over all edges $E_0$ with $\mathcal M(E_0)\neq -1$ in the following three coordinate systems: $(x,y)$, $(x,x+y)$ and $(y, x+y)$.  }

Let $\rho_0>\epsilon$ be a small number chosen so that the $\rho_0$ neighborhoods of 
all the roots $r_*$ will not overlap. 
If $|r-r_*| \geq\rho_0$ for all roots $r_*$ and if $\epsilon$ sufficiently small,  then  
 \begin{align}\label{cool}
|P(x,y)|= |P_0(x,y)| \sim_{\rho_0} |P_{0,E_0}(x,y)|\sim |x^{p_0}y^{q_0}|\sim x^{p_0+m_0{q_0}},
 \end{align}
which means the edge $E_0$ dominates
 $P_0(x,y)$. Similarly, for any preselected $a, b \geq 0$, (\ref{DE01}) still holds.  
 We refer such regions as `good' regions defined by the edge $E_0$ and denote them by $U_{0, g, E_0, j}$'s. 
It is of significance to observe that one can enlarge each $U_{0, g, E_0, j}$ by decreasing the value of $\rho_0$, while 
$|P(x,y)|\sim  |x^{p_0}y^{q_0}|$ still holds but with the new implicit constant depending on the new $\rho_0$. 
However, one should not worry about the dependence on those $\rho_0$, for they will be chosen to 
rely only on the original function $P$. We will use this observation to further decompose  $U_{0, g, E_0, j}$ into rectangular boxes later.

 For each root $r_0\in \{r_*\}$, there is an associated `bad' region defined by $y=rx^{m_0}$ where $|r-r_0|< \rho_0$ and $0<x<\epsilon$.
 We say that this is a bad region defined by the edge $E_0$ (and the root $r_0$).
 Each bad region is carried to the next stage of iteration via change of variables. Set
 \begin{align*}
 \begin{cases}
 x =x_1,
 \\
 y =x_1^{m_0}(r_0+y_1).
 \end{cases}
 \end{align*}
 Notice that this bad region is contained in the set where $|y_1 |= |r-r_0| <\rho_0$, so that in $(x_1,y_1)$ coordinates, the bad region is now $0<x_1<\epsilon$ and $-\rho_0<y_1<\rho_0$. 
Set
 \begin{align*}
 P_1(x_1,y_1) = P_0(x_1, r_0 x_1^{m_0} +y_1x_1^{m_0}), 
 \end{align*}
 and for any choice of $\epsilon_1 > \rho_0$, let
 \begin{align*}
 U_1 =\{(x_1,y_1): 0<x_1<\epsilon_1, |y_1|< \epsilon_1 \}.
 \end{align*}
 Then $[U_1, P_1, (x_1,y_1)]$ is a standard triple and same arguments can be applied again. 
Notice that the edge $E_0$ in $\mathcal N(P)$ is ``collapsed'' into the leftmost vertex of $\mathcal N(P_1)$. More precisely,  
 if $(p_{1,l}, q_{1,l}) $ is the leftmost vertex of $\mathcal N(P_1)$, then 
 \begin{align}\label{key88}
(p_{1,l}, q_{1,l}) = (p_0 +m_0q_0, s_0).
 \end{align}
 This is a crucial bookkeeping identity for understanding the behavior of $P(x,y)$ in later stages of the iteration. 
Now we repeat the previous arguments. If either a vertex or an edge is dominant, then 
\begin{align*}
|P_1(x_1,y_1)| \sim |x_1^{p_1} y_1^{q_1}|
\end{align*} 
 for some vertex $(p_1,q_1)$ of $\mathcal N(P_1)$. Otherwise
there is an edge $E_1\in\mathcal E(P_1)$ and a non-zero root $r_1$ of $P_{1, E_1}(1, y_1)$ together with a neighborhood 
$(r_1 -\rho_1,r_1+\rho_1)$ that define a bad region. 
 Change variables to
  \begin{align*}
 \begin{cases}
 x_1 =x_2
 \\
 y_1 =x_2^{m_1}(r_1+y_2)
 \end{cases}
 \end{align*}
 and set 
 \begin{align*}
 P_2(x_2,y_2)  = P_1( x_2, x_2^{m_1}(r_1+y_2)).
 \end{align*}
Now we iterate the above argument. In the $k$-th stage of iteration, if either a vertex or an edge is dominant, then
\begin{align}\label{ESk}
|P(x,y)| = |P_k(x_k,y_k)|\sim |x_k^{p_k}y_k^{q_k}|
\end{align}
and 
\begin{align}\label{DEk}
|\p_{x_k}^a\p_{y_k}^b P_k(x_k,y_k) |\lesssim \min\{ 1, |x_k^{p_k-a}y_k^{q_k-b}| \}.
\end{align}
Here $(p_k,q_k)$ is a vertex of $\mathcal N(P_k)$ and 
\begin{align*}
\begin{cases}
x_{j-1} = x_{j},
\\
y_{j-1} = (r_{j-1}+y_j)x^{m_{j-1}},
\end{cases}
\end{align*}
for $1 \leq j\leq k$, i.e.,
\begin{align*}
\begin{cases}
x = x_0 =\dots = x_k
\\
y = y_0 = r_0x_0^{m_0}+r_1x_0^{m_0+m_1}+\dots +r_{k-1}x_0^{m_0+\dots+m_{k-1}}+y_kx^{m_0+\dots +m_{k-1}}.
\end{cases}
\end{align*}
The iterations above form a tree structure and each branch of this tree yields a chain of standard triples  
\begin{align}\label{chain70}
[U,P,(x,y)]=[U_0,P_0,(x_0,y_0)]\to[U_1,P_1,(x_1,y_1)]\to\dots\to [U_k,P_k,(x_k,y_k)]\to\cdots.
\end{align}
One may hope each such chain is finite, but unfortunately, this is not necessarily the case. For those that terminate after finitely many steps, say at the $n$-th stage of iteration, it must be the case that no bad regions are generated by $[U_n,P_n, (x_n,y_n)]$, and so we obtain Theorem \ref{rs90} with the corresponding change 
of variables given by (\ref{gamma2x}).  

For each of those branches that does not terminate,
one can show that (see Lemma 4.6 \cite{X2013}), there exists  $n_0\in\mathbb N$ such that for $n\geq n_0$, 
 $\mathcal N(P_n)$ has only one compact edge $E_n$ and $P_{n,E_n}(x_n,y_n)= c_n(y_n-r_nx_n^{m_n})^{s_{n_0}}$,
 where $s_{n_0}$ is the order of the root $r_{n_0}$. On such a branch, one should instead perform  
 the following change of variables:
 $$
y_{n_0} =  y_{n_0+1}x_{n_0}^{m_{n_0}}+ \sum\limits_{k=n_0}^{\infty} r_kx_{n_0}^{m_{n_0}+m_{n_0+1}+\cdots+m_k}.
 $$
Under this alternate change of variables, iteration along this branch stops immediately at stage $n_0+1$. This corresponds to the change of variables in (\ref{gamma1x}). 
Since the total number of branches in the iteration is bounded above (see Lemma 4.5  in \cite{X2013}), the modified algorithm fully terminates after finitely many steps. In particular, the total number of good regions are also finite. 

We now give some explanation of the subindices $({n,g, {\bsb \alpha}, j})$ in Theorem \ref{rs90}. 
The letter `$n$' represents the stage of iterations, `$g$' indicates the region is `good',  `${\bsb \alpha}$' contains the information necessary to make the change of variables 
from $[U_0,P_0,(x_0,y_0)]$ to a specific $[U_n,P_n,(x_n,y_n)]$ and `$j$' is used to list the all the `good' regions generated by  $[U_n,P_n,(x_n,y_n)]$. 
The cardinality of the tuples $({n,g, {\bsb \alpha}, j})$ is finite and depends on the original function $P$.
We often use $U_{n,g}$ to represent  $U_{n,g}= U_{n,g,\bsb\alpha,j}$ for some $\bsb\alpha$ and some $j$. 

The identity (\ref{key88}) is very useful to help estimate $P(x,y)$ in higher stages of the iteration. Indeed, 
consider the $k$-th stage of iteration and let $U_{k,g, {\bsb \alpha}, j}$ be a good region. Then 
$|P(x,y)| =|P_k(x_k,y_k)| \sim |x_k^{p_k} y_k^{q_k}|$, for some vertex $V_k=(p_k,q_k)$ of $\mathcal N(P_k)$. 
For $0 \leq j \leq k-1$,
let $(p_{j, l},q_{j,l})$ be the leftmost vertex of $\mathcal N(P_j)$, and furthermore let $(p_j,q_j)$ be the left vertex of the
edge $E_j$ and let $r_j$ be the root that governs the next stage of iteration.
Assume $s_j$ is the order of the root  $r_j$. 
 We claim that for $0\leq j\leq k-1$, 
\begin{align}\label{ind1}
\begin{cases}
p_{j+1,l} = p_{j} +q_{j}m_j , 
\\
 q_{j+1,l}  =s_j \leq q_j,
\\
p_{j}+q_jm_j \leq p_{j,l}+q_{j,l}m_j,
\\
\end{cases}
\end{align}
and if $L_{m_k}$ is a supporting line of $\mathcal N(P_k)$ through $V_k$, then 
\begin{align}\label{IT5}
p_{k}+m_k q_k \leq p_{k,l}+m_k q_{k,l}  \leq  p_0+m_0q_0 +s_0\sum_{1\leq j \leq k-1} m_j.
\end{align}
Indeed the first two identities in (\ref{ind1}) are of the same nature as the one in (\ref{key88}). 
The third one comes from the fact that the vertex $(p_{j,l},q_{j,l})$ lies on or above $E_j$, which is the
 supporting line through $(p_j,q_j)$ of slope $-1/m_j$. 
Iterating (\ref{ind1}) yields (\ref{IT5}).

\subsection{A smooth partition}
Let $U$ be a small neighborhood of the origin such that one can apply Theorem \ref{rs90} to it. 
 Divide $U$ (up to a set of measure zero) into four different regions: $U_E$, $U_N$, $U_W$ and $U_S$
 (representing the east, north, west and south portions), defined by 
\begin{align}
\begin{cases}
U_E= \{(x,y)\in U: x>0, -Cx < y < Cx\}
\\
U_W = -U_E
\\
U_N = \{(x,y)\in U: C |x| < y \}
\\
U_S= -U_N.
\end{cases}
\end{align}
We choose the above constant $C>0$ such that $2^{-10}C$ is greater than the absolute value of any root of 
$P_E(1,y)$ or $P_E(-1,y)$, where $E$ is the edge of slope $-1$. 
We shall then define smooth functions $\phi_E$, $\phi_N$, $\phi_W$ and $\phi_S$ whose supports are contained in 
$\frac 2 3 U$ and such that for any function $\Psi$ supported in $\frac 1 2 U$, the
following holds :
\begin{align*}
\Psi(x,y) = (\phi_E(x,y)+\phi_N(x,y)+\phi_W(x,y)+\phi_S(x,y))\Psi(x,y) \ \ {\rm for}\,\, (x,y)\neq 0.   
\end{align*} 
Moreover, $\phi_E$ is essentially supported in $U_E$, in the sense that $\supp \phi_E\subset  U^*_E$, where 
\begin{align}
U_E^* = \{(x,y)\in U: -2Cx <y <2C x\},
\end{align} 
and similarly for $\phi_N$ and so on.  

In what follows, we focus on $\phi_E$ and $U_E$. Similar results for $U_W$, $U_N$ and $U_S$ can be obtained
by changing $x$ to $-x$, switching $y$ and $x$, and making both changes, respectively. 
Notice $|y|\lesssim |x|$ for $(x,y)\in U_E^*$, and by shrinking $U$ if necessary, one has $m_0 \geq 1$ if $U_{n,g,{\bsb\alpha},j}$ has nonempty intersection with $U_E^*$, where $m_0$ is the exponent of the first term of $\gamma_n(x)$ in (\ref{gamma1x}) or (\ref{gamma2x}). 
We will partition $\phi_E$ into a sum of smooth functions $\phi_\mathcal R$
such that each $\phi_\mathcal R$ is essentially supported in a box $\mathcal R$ which is essentially contained in one $U_{n,g,{\bsb\alpha},j}$.  
Consequently $P(x,y)$ still behaves like a monomial in $\supp \phi_\mathcal R$ and Theorem \ref{local} is applicable. 
The rest of this section is devoted to constructing this smooth partition. In what follows, let $U_{n,g} =  U_{n,g,{\bsb\alpha},j}$ for some 
${\bsb\alpha},j$. 

Firstly, 
in the algorithm above, we have the flexibility to adjust the ``sizes'' of $U_{n,g}$ by modifying the values of $\rho_n$ at each stage of iteration. 
In particular, one can enlarge 
each $U_{n,g}$ to $U^*_{n,g}$, 
  i.e.,
$
U_{n,g} \subset U^*_{n,g}
$,
so that the estimate $|P(x,y)|\sim |x_n^{p_n}y_n^{q_n}|$ still holds for $(x,y)\in U^{*}_{n,g}$ 
 with a different implicit constant. This is done by moving up the upper boundary of $U_{n,g}$ and moving down the 
lower boundary of $U_{n,g}$ simultaneously. The following example illustrates this idea.
\begin{example}\label{EX01}
For simplicity, we assume $U_{n,g} $ is defined by an edge $E_n$ and is lying in the first quadrant. 
Then $U_{n,g} $ is equal to the intersection of $U$ with the curved triangular region 
given by 
\begin{align}
\xi(x) := \gamma_n(x) +a x^{m_0+\dots +m_n} \leq y \leq  \gamma_n(x) +Ax^{m_0+\dots +m_n} :=\eta(x)
\end{align}
for some constants $0<a<A$.  
Recall that the good regions defined by $E_n$ are the regions given by $y_n = r x_n^{m_n}$ with
$r$ lying outside the $\rho_n$ neighborhoods of the roots $\{r_{n,j}\}_{j}$ (listed in the increasing order) of $P_{n, E_n}(1,y_n)$. Here $m_n =\mathcal M(E_n)$ and 
$P_{n, E_n}$ is the restriction of $P_n$ to the edge $E_n$. Then there is a $j_0$ such that $a = r_{n,j_0}+\rho_n$ and $A= r_{n,j_0+1}-\rho_n$.  
 One can then adjust $\rho_n$ to $\rho_n^* =\frac {\rho_n}  2 $, 
set $a^* = r_{n,j_0}+\rho^*_n$ and $A^*= r_{n,j_0+1}-\rho^*_n$ 
and define $ U^{*}_{n,g}$ to be the intersection of $U$ with 
$$
\xi^*(x): = \gamma_n(x) +a^* x^{m_0+\dots +m_n} \leq y \leq  \gamma_n(x) +A^*x^{m_0+\dots +m_n}:=\eta^*(x).
$$
In these new enlarged good regions, $P(x,y)$ still behaves like a monomial, i.e. (\ref{ESk}) and (\ref{DEk}) still hold but with different implicit constants (depending on $\rho_n^*$). 
\end{example}

The next step is to cover $ U_{n,g}$ by a collection of rectangular boxes $\mathcal R$ such that for some fixed constant $\tau>1$ and for each $\mathcal R$, the dilation 
 $\tau\mathcal R$ lies in $ U^{*}_{n,g}$. 
  We shall illustrate this covering in the context of Example \ref{EX01} and briefly comment on how to handle other cases at the end.   Dyadically decompose $U_{n,g}$ as
\begin{align*}
U_{n,g}(\sigma) = \{(x,y)\in U_{n,g}: x\in I_\sigma\}.  
\end{align*} 
where $I_\sigma= [\sigma,2 \sigma)$. 
Set
\begin{align}
\begin{cases}\label{sim07}
&\Delta x = \frac {\sigma ^{m_0+\dots +m_n} \cdot  \sigma^{1-m_0}} {K},
\\
&\Delta y = \sigma ^{m_0+\dots +m_n},
\end{cases} 
\end{align}
where $K\in\mathbb N$ is a constant (to be determined later) depending on $P$ but independent of $\sigma$. Roughly speaking, $\Delta x $ and $\Delta y$ are proportional   
to the measures of a generic $x$-cross section and a generic $y$-cross section of $U_{n,g}(\sigma)$, respectively. 
Equally divide $I_\sigma$ into subintervals $\{I_{\sigma,h}\}_{0\leq h \leq H}=\{[\sigma_{h}, \sigma_{h+1})\}_{0\leq h \leq H}$ 
of length $\Delta x$, where 
$\sigma_h = \sigma+ h\Delta x$ and $H= \frac \sigma {\Delta x}-1 $.  Let $\mathcal R_{\sigma, h}$ be the rectangle whose sides are 
parallel to the coordinate axes with $(\sigma_{h}, \xi(\sigma_h))$ and $(\sigma_{h+1}, \eta(\sigma_{h+1}))$ as its
 lower left and upper right vertices.  
Let $\mathcal R^*_{\sigma, h}$ be defined similarly but with the lower left and upper right vertices replaced by 
$(\sigma_{h-1}, \xi(\sigma_{h-1}))$ and $(\sigma_{h+2}, \eta(\sigma_{h+2}))$. It is obvious that 
the union of $\{\mathcal R_{\sigma, h}\}_{0\leq h\leq H}$ covers $U_{n,g}(\sigma)$ and for notational convenience, we will use 
$\{\mathcal R\}_{\mathcal R\in U_{n,g}(\sigma)}$
to denote this collection of rectangular boxes. 
Beyond this,  we also have the following claim.
\begin{claim}\label{CL01}
One can choose $K$ independent of $\sigma$ such that $\mathcal R^*_{\sigma, h}\subset U_{n,g}^*$ for all $0\leq h \leq H$.  
In addition, there is a constant $\tau>1$ independent of $\sigma$ such that the dilation $\tau \mathcal R_{\sigma, h} \subset \mathcal R^*_{\sigma, h}$. 
\end{claim}
\begin{proof}[Proof of Claim \ref{CL01}]
Label the vertices of $\mathcal R_{\sigma, h}$ as $A, B, C, D$ in a counter-clockwise fashion beginning with $A$ as the lower left vertex. Following the same process, label the vertices of 
$\mathcal R^*_{\sigma, h}$ by $A', B', C', D'$. Then both $A$ and $A'$ lie on the lower boundary of $U_{n,g}$ and both $C$ and $C'$
lie on the upper boundary. To prove the first part, it suffices to show that $B' $ and $D'$ are both in $U_{n,g}^*$, which amounts to proving that 
\begin{align}\label{TO39}
\xi(\sigma_{h-1}) > \xi^{*}(\sigma_{h+2}) \q {\rm and}\q \eta^{*}(\sigma_{h-1})> \eta(\sigma_{h+2})
\end{align} 
for an appropriate choice of $K$. 
Recall that $\gamma_n(x)$ equals $r_0x^{m_0}$ plus higher-order terms. One can choose $\epsilon$ small such that
 \begin{align}
 \frac 1 2 m_0 |r_0| x^{m_0-1} \leq 
 |\xi'(x)|, |\eta'(x)| \leq 2m_0 |r_0| x^{m_0-1}\q {\rm for \,\, all } \q 0<x <\epsilon. 
 \end{align} 
 In particular, one has
 \begin{align}\label{TU90} 
c_0\sigma^{m_0-1}\leq  |\xi'(x)|, |\eta'(x)| \leq C_0\sigma^{m_0-1}\q {\rm  for\,\, all} \q \frac 1 2 \sigma < x\leq 2\sigma
 \end{align}
 with $c_0 =  m_0 |r_0| 2^{-m_0}$ and  $C_0 =2m_0 |r_0| 2^{m_0-1}$. 
Then by the Mean Value Theorem and the above estimates for $|\xi'(x)|$, one has 
$$
\xi(\sigma_{h+2})- \xi(\sigma_{h-1}) \leq C_0 \sigma^{m_0-1} (\sigma_{h+2} -\sigma_{h-1})
= \frac {3C_0} {K} \sigma^{m_0+\cdots+m_n}. 
$$
Notice also that
$$
\xi(\sigma_{h+2})- \xi^*(\sigma_{h+2}) =  \frac {\rho_n} 2  \sigma_{h+2}^{m_0+\cdots+m_n}\geq 
\frac {\rho_n} { 2} \sigma^{m_0+\dots+m_n}, 
$$
because one has the trivial estimates $\sigma_{h+2} \geq  \sigma_{0}= \sigma$.  
The first estimate in (\ref{TO39}) follows by comparing the above two estimates and by choosing
$$
K > \frac {6C_0} {\rho_n}. 
$$
To verify the second estimate in (\ref{TO39}), similar arguments yield
$$
\eta (\sigma_{h+2}) -\eta (\sigma_{h-1}) \leq C_0 \sigma^{m_0-1} (\sigma_{h+2} -\sigma_{h-1}) =\frac {3C_0} {K} \sigma^{m_0+\cdots+m_n} 
$$
and 
$$
\eta^*(\sigma_{h-1})-\eta (\sigma_{h-1}) =  \frac {\rho_n} 2  \sigma_{h-1}^{m_0+\cdots+m_n}\geq 
\frac {\rho_n} 2 2^{-(m_0+\cdots+m_n)}  \sigma^{m_0+\cdots+m_n}
$$
for $\sigma_{h-1} \geq \sigma_{-1} \geq \frac 1 2 \sigma_0 = \frac 1  2 \sigma$. 
Then the second (and the first) estimate of (\ref{TO39}) follows by choosing 
$$
K > \frac {6C_0}{\rho_n}2^{m_0+\dots+m_n}. 
$$
We turn to the second part of the claim and let $K$ be a fixed constant satisfying the above requirements. Notice 
$$
{\rm dist}(AB, CD) = |\eta(\sigma_{h+1}) -\xi(\sigma_{h})| \leq  
 |\eta(\sigma_{h+1}) -\xi(\sigma_{h+1})|+ |\xi(\sigma_{h+1}) -\xi(\sigma_{h})| 
$$ 
which is bounded by $C\Delta y$ for some constant $C$.  Here ${\rm dist}(\cdot, \cdot)$ denotes the distance between two parallel lines.
 Since 
$$
{\rm dist}(AB, A'B') = |\xi(\sigma_{h}) -\xi(\sigma_{h-1})| \geq c_0 \sigma^{m_0-1}  \Delta x =\frac {c_0} K\Delta y, 
$$ 
$$
{\rm dist}(CD, C'D') = |\eta(\sigma_{h+2}) -\eta(\sigma_{h+1})| \geq c_0 \sigma^{m_0-1}  \Delta x =\frac {c_0} K\Delta y, 
$$ 
and 
$$
{\rm dist}(AD, BC) = {\rm dist}(AD, A'D') = {\rm dist}(BC, B'C') = \Delta x,
$$
it is obvious that there is a $\tau >1$ independent of $\sigma$ such that 
$\tau \mathcal R_{\sigma, h} \subset \mathcal R^*_{\sigma, h}$. 
\end{proof}
We have now finished the second step. At this point we will briefly discuss the construction of similar coverings for cases other than Example \ref{EX01}. Since the case when $U_{n,g}$ is  defined by an edge lying in the fourth quadrant
is very similar, we focus on the case when $U_{n,g}$ defined by a vertex $V_n$ and also assume it lies in the first quadrant. 
Then $U_{n,g}$ is equal to the intersection of $U$ with the region
\begin{align}
\gamma_n(x) +C_r x^{m_0+\dots +m_{n-1}+m_{n,r }}\leq   y \leq  \gamma_n(x) +C_l x^{m_0+\dots +m_{n-1}+m_{n,l }},
\end{align}
where $-1/m_{n,l}$ and $-1/m_{n,r}$ the are slopes of the edges to the left and right of $V_n$ and $C_r, C_l>0$ are finite constants. 
For convenience, we write $m_n = m_{n,l}$. 
Here we also need to dyadically decompose in $y$. Let 
$\sigma, \rho$ be dyadic numbers and set 
\begin{align}
U_{n,g}(\sigma,\rho)= \{(x,y)\in U_{n,g}: \sigma \leq x <2\sigma,  \rho x^{m_{n}}\leq y_n <2\rho x^{m_{n}} \},  
\end{align}
where $y = \gamma_n(x) +y_nx^{m_0+\dots +m_{n-1}}$.
The same arguments employed in Example \ref{EX01} can be then applied to $U_{n,g}(\sigma,\rho)$, 
yielding a collection of rectangular boxes $\{\mathcal R\}_{\mathcal R\in {U_{n,g}}(\sigma,\rho)}$ whose union covers
$U_{n,g}(\sigma,\rho)$. The dimensions of each $\mathcal R$ are roughly $\Delta x \times \Delta y$, where in the case $n\geq 1$ 
\begin{align}\label{sim08}
 \begin{cases}
 &\Delta x = \frac { \rho\sigma^{m_0+\dots +m_n} \cdot \sigma^{1-m_0} }{K},
 \\
& \Delta y = \rho\sigma^{m_0+\dots +m_n},
\end{cases}
 \end{align}
and in the case $n=0$
\begin{align}\label{sim18}
 \begin{cases}
 &\Delta x = \frac {\sigma}{K},
 \\
& \Delta y = \rho\sigma^{m_0}. 
\end{cases}
 \end{align}
 Here $K$ is a constant independent of $\sigma$ and $\rho$ which fills the same function as the constant in Example \ref{EX01}. 
 To be precise, $K\in\mathbb N$ is selected such that
there is a constant $\tau>1$ independent of $\sigma$ and $\rho$ such that $\tau \mathcal R$ lies in $U_{n,g}^*(\sigma,\rho)$ for each
$\mathcal R\in U_{n,g}$. 
Here $U_{n,g}^*(\sigma,\rho)$ is an enlarged version of $U_{n,g}(\sigma,\rho)$ defined similarly to the analogue in Example \ref{EX01}.

For notational convenience, in the case when $U_{n,g}$ is defined by an edge, we shall also use $U_{n,g}(\sigma,\rho)$ to denote $U_{n,g}(\sigma)$ and keep in mind that $\rho\sim 1$, even though we do not need decompose $U_{n,g}(\sigma)$ in the $y$-direction. 
 
Let us pause briefly to observe that collection operators associated to the boxes $\{\mathcal R\}_{\mathcal R\in U_{n,g}(\sigma,\rho)}$ will exhibit certain orthogonality properties which we will discuss now.
 Indeed, let $\pi_x$ and $\pi_y$ denote the othogonal projections of $\mathbb R^2$ onto the $x$- and $y$-axis, respectively. 
Then there exists a positive constant $C\in\mathbb N$
 independent of $\sigma$ and $\rho$ such that 
 \begin{align}\label{OT67}
\sum_{\mathcal R\in U_{n,g} (\sigma,\rho)} \Id_{\pi_x(\mathcal R)}(x)  \leq C \q 
\textrm{and}\q \sum_{\mathcal R\in U_{n,g} (\sigma,\rho)} \Id_{\pi_y(\mathcal R)}(y) \leq C. 
 \end{align}
Notice that $C$ can be also independent of each $U_{n,g}$, for the cardinality of the collection of good regions $U_{n,g,{\bsb\alpha}, j}$ is finite. 
This orthogonality is important in the $L^2\times L^2 \times L^2$ estimates, 
but not in the $L^\infty\times L^\infty \times L^\infty$ here. 
What will be useful in its place is the cardinality estimate
\begin{align}\label{card01}
\# \{\mathcal R\}_{\mathcal R\in U_{n,g}(\sigma,\rho)} \sim \frac \sigma  {\Delta x}. 
\end{align}
 
Returning to the resolution algorithm, the third step is to complete the smooth partition. Let $\phi_\mathcal R$ be a non-negative 
smooth function essentially supported in $\mathcal R$, in the sense that it satisfies 
 \begin{align}
\mathcal R\subset  \supp \phi_{\mathcal R} \subset \tau \mathcal R 
 \end{align}
  and 
 \begin{align}
|\partial_x^a \partial_y ^b  \phi_{\mathcal R}(x,y)| \lesssim |\Delta x|^{-a} | \Delta y|^{-b}, \q \textrm {for} \q 0\leq a, b \leq 2.
 \end{align}
 Consequently,  
 \begin{align}
  \displaystyle  \sum_{\sigma}\sum_{\rho}\,\,\,\sum_{\mathcal R\in U_{n,g} (\sigma,\rho)} \phi_{\mathcal R} (x,y)\neq 0\,,\q \textrm {for all} \q(x,y)\in U_{n,g}.
 \end{align}
We already know that $|P(x,y)|=|P_n(x_n,y_n)| \sim   |x_n^{p_n}y_n^{q_n}|$, 
but we can also control upper bounds of certain derivatives of $P(x,y)$ for $(x,y)\in \tau \mathcal R$. 
Indeed, by the chain rule, $ \p_{y}^b P(x,y)$ is equal to 
 \begin{align}
\p_{y_n}^b P(x,y) \left( \frac{ \p y} {\p{y_n}}\right) ^b =  \p_{y_n}^b P(x,y) \cdot  x^{-b(m_0+\dots+m_{n-1})}.
 \end{align}
By (\ref{DEk}), we have   
for $(x,y)\in \tau \mathcal R$, 
   \begin{align}
| \p_{y}^b P(x,y) | &\lesssim |x_n^{p_n}y_n^{q_n}| \cdot \Delta y ^{-b}  \sim   |P_n(x_n,y_n)| \cdot \Delta y ^{-b}. 
 \end{align}
 Similarly, viewing $y$ as a function of $x$, one has 
 \begin{align*}
 \p_{x} P(x,y) = (\p_{x} P)(x,y)+  (\p_{y} P)(x,y) \frac{ \partial y }{\p x}. 
 \end{align*}
 Since $|\frac {\p y}{\p x}| \lesssim |x|^{m_0-1} $,  $| \p_{x} P(x,y)| $ is bounded by 
  \begin{align*}
& |x_n^{p_n} y_n^{q_n}| |x_n|^{-1} + |x_n^{p_n}y_n^{q_n}| |x^{m_0+\dots+m_{n-1}+m_n}y_n|^{-1}\cdot  |x|^{m_0-1} 
 \end{align*}
 and it is obvious that the later is dominant. Thus $ |\p_{x} P(x,y)|\lesssim |x_n^{p_n}y_n^{q_n}| \Delta x^{-1 }$. 
Analogous calculations for $a =0, 1,$ and $2$ yield that   
  \begin{align*}
| \p_{x}^a P(x,y)| \lesssim  |x_n^{p_n}y_n^{q_n}| \cdot  \Delta x^{-a}\sim |P_n(x_n,y_n)|\cdot  \Delta x^{-a}
 \end{align*}
 for $(x,y)\in \tau\mathcal R$.  
 Finally, we associate the function 
 \begin{equation}
\dfrac {\phi_\mathcal R}{ \sum_{(n,g,{\bsb\alpha},j)}\sum_{\sigma}\sum_\rho\,\,\,\sum_{\mathcal R\in U_{n,g,{\bsb\alpha},j} (\sigma,\rho)} \phi_{\mathcal R} }\cdot 
 \phi_E \label{partou} 
 \end{equation}
to ${\mathcal R}$ (and, for convenience, redefine $\phi_{\mathcal R}$ to equal \eqref{partou}) and therefore 
 \begin{align}
  \phi_E(x,y) = \sum_{(n,g,{\bsb\alpha},j)}\sum_{\sigma}\sum_\rho\,\,\,\sum_{\mathcal R\in U_{n,g,{\bsb\alpha},j} (\sigma,\rho)} \phi_{\mathcal R}(x,y). 
 \end{align}

%
%
%
%
%
%
%
%
%
%
%
%
%
%
%

\section{Proof of the main theorem}\label{PF03}
As  was done in the previous section, we decompose $\phi$ into the sum of four functions supported in the east, north, west and south regions respectively, i.e.  
$$
\phi(x,y) =\phi_E(x,y) +\phi_N(x,y)+\phi_W(x,y)+\phi_S(x,y)\,\,{\rm for}\,\, (x,y)\neq 0.
$$
The trilinear form can be then written as 
$$
\Lambda(f,g,h) = \Lambda_E(f,g,h) +\Lambda_N(f,g,h)+\Lambda_W(f,g,h)+\Lambda_S(f,g,h)   
$$
where 
\begin{align}\label{TR70}
 \Lambda_E(f,g,h) =\iint e^{i\lambda S(x,y)}f(x)g(y)h(x+y)\phi_E(x,y) dxdy,
\end{align}
and so on.  We will focus on $ \Lambda_E(f,g,h)$, for the others can be handled similarly. 
Apply the smooth partition of unity from the previous section built from the function 
$P(x,y) =DS(x,y)$; we can write 
\begin{align*}
\phi_E = \sum_{(n,g, {\bsb\alpha}, j)} \sum_{\sigma, \rho} \sum_{\mathcal R\in U_{n,g,{\bsb\alpha}, j}(\sigma,\rho)} \phi_{\mathcal R},
\end{align*}
where $\sigma$ and $\rho$ are dyadic numbers and conclude that
$$
 \Lambda_E(f,g,h) =\sum_{(n,g, {\bsb\alpha}, j)} \sum_{\sigma, \rho} \sum_{\mathcal R\in U_{n,g,{\bsb\alpha}, j}(\sigma,\rho)} 
 \Lambda_{\mathcal R}(f,g,h)
$$
where $ \Lambda_{\mathcal R}(f,g,h)$ is defined as in (\ref{TR70}) with $\phi_E$ replaced by $\phi_\mathcal R$. 
Since the number of $  U_{n,g,{\bsb\alpha}, j}$ is finite, it suffices 
to prove the desired estimate in one $U_{n,g,{\bsb\alpha}, j}$. We will focus on $U_{k,g}= U_{k,g,{\bsb\alpha}, j}$ for some 
${\bsb\alpha}, j$, where we used $k$ as the subindex to indicate the stage of iteration. 
Write 
$$
 \Lambda_{k,g}(f,g,h) = \sum_{\sigma, \rho} \sum_{\mathcal R\in U_{k,g}(\sigma,\rho)} 
 \Lambda_{\mathcal R}(f,g,h). 
$$
In what follows, we assume $\|f\|_\infty =\|g\|_\infty =\|h\|_\infty=1$ and use the shorthand
$$
\Lambda_* = \Lambda_*(f,g,h) 
$$
(where $*$ represents any subindex) and  
$$
\|\Lambda_*\| =\sup|\Lambda_* (f,g,h)|,
$$  
where the supremum ranges over all functions with $\|f\|_\infty =\|g\|_\infty =\|h\|_\infty=1$. We will also use the notation
$$
\Lambda_{k,g}[\sigma,\rho] = \sum_{\mathcal R\in U_{k,g}(\sigma,\rho)} 
 \Lambda_{\mathcal R}.
$$ 
Notice that for $(x,y) \in U_{k,g}$ one has 
\begin{align}
|DS(x,y)| \sim |x_k^{p_k}y_k^{q_k}|
\end{align}
where 
\begin{align}\label{yn}
\begin{cases}
x=x_0=\dots=x_{k},
\\
y =y_0 = \gamma_k(x)+y_kx^{m_0+\dots+m_{k-1}}.
\end{cases}
\end{align}
Here 
\begin{align}
\gamma_k(x) =\sum_{j=0}^* r_0x^{m_0+\dots+m_{j}}
\end{align}
with $*=k-1$ or $\infty$; see (\ref{gamma1x}) and (\ref{gamma2x}). 
Notice that $m_0\geq 1$ since $(x,y) \in \supp \phi _E$.  
For the edge in $\mathcal N(DS)$
corresponding to $m_0=1$, its left and right verticies are $(\alpha, n-3 -\alpha)$ and $(n-3-\beta, \beta)$, respectively 
(where it is possible that these vertices are equal). 
The basic estimate on the rectangles ${\mathcal R}$ is as follows.
 \begin{lemma}\label{LM6}
For each $\mathcal R\in U_{k,g}(\sigma,\rho)$
\begin{align}\label{SZ01}
\| \Lambda_\mathcal R \|  \lesssim \Delta x\cdot \Delta y 
\end{align}
and 
 \begin{align}\label{OS01}
\| \Lambda_\mathcal R \| \lesssim |\lambda|^{-\frac 1 4}  \left(\inf_{(x,y)\in \phi_\mathcal R}|P(x,y)|\right)^{-\frac 1 4}
\Delta x^{\f 1 p} \Delta y ^{\f 1 q}  (\Delta x+\Delta y)^{\f 1 r} 
 \end{align}
  for all $(p,q,r)\in\mathcal A$. 
 \end{lemma}
\begin{proof}
Recall that the support of $\phi_\mathcal R $ is contained in a rectangular box of dimensions $\Delta x \times \Delta y$, with $\Delta x $ and $\Delta y$ given by
(\ref{sim07}), or (\ref{sim08}) or (\ref{sim18}). Then (\ref{SZ01}) follows by the triangle inequality, passing absolute values into the integral of $\Lambda_\mathcal R $. 
 In estimating
$\| \Lambda_\mathcal R \|  $, we can always assume $f$, $g$ and $h$ are supported on intervals of lengthes $\Delta x$, $\Delta y$ and 
$(\Delta x+\Delta y)$ respectively. Notice for $0\leq a, b \leq 2$, 
$$
|\p_x^{a}\p_y^{b} \phi_{\mathcal R}(x,y)| \lesssim \Delta x^{-a} \Delta y^{-b},
$$ 
$$
|\p_x ^{a} P(x,y) \lesssim |P(x,y)| \Delta x^{-a}
$$
and 
$$
|\p_y ^{b} P(x,y) \lesssim |P(x,y)| \Delta y^{-b}.
$$
For $(x,y)\in\supp \phi_\mathcal R$, $|P(x,y)|$ is essentially a constant in the sense that 
$$
\inf_{(x,y)\in \supp \phi_\mathcal R}|P(x,y)| \sim \sup_{(x,y)\in \supp \phi_\mathcal R} |P(x,y)| \sim |P(x,y)|\sim |x_k^{p_k}y_k^{q_k}|.
$$
Corollary \ref{LC90} then yields  
\begin{align}
\| \Lambda_\mathcal R \| \lesssim |\lambda|^{-\frac 1 4}  \left(\inf_{(x,y)\in \phi_\mathcal R}|P(x,y)|\right)^{-\frac 1 4}
\|f\|_p\|g\|_q\|h\|_r,
 \end{align}
and (\ref{OS01}) follows by noticing $\|f\|_\infty =\|g\|_\infty =\|h\|_\infty=1$ and using the size of the supports of $f$, $g$ and $h$. 
\end{proof}

The next step is to sum over $\| \Lambda_\mathcal R \|$ for $\mathcal R$ in $U_{k,g}(\sigma,\rho)$, i.e. to estimate 
$\Lambda_{k,g}[\sigma,\rho] $. 
To do so set 
\begin{align}\label{theta0}
\theta =\max\{\kappa , \frac  n 2-2\}
\end{align}
 and consider the following four cases: 
\begin{enumerate}
\item[]  \textbf{Case 1}: $k=0$, 
\item[]
 \textbf{Case 2}: $k\geq 1$ but $(m_0, r_0)\neq (1, -1)$, 
 \item[] \textbf{Case 3}: $k= 1 $ and $(m_0, r_0)= (1, -1)$,
  \item[]  \textbf{Case 4}: $k\geq 2 $ and $(m_0, r_0)= (1, -1)$. 
\end{enumerate}
 In the arguments that follow, we will use (\ref{OS01}) for triples $(p,q,r)$ equal to $(\f 8 3, 2, \f 8 3)$ in Cases 1 and 2 and $(\f 8 3,  \f 8 3, 2)$ in Cases 3 and 4. 
 
\subsection{Case 1 $k=0$}
In this case, the support of $\mathcal R$ is in good region $U_{0,g}$ and $|P(x,y)| \sim |x^{p_0}y^{q_0}|$ for $(x,y)\in U_{0,g}$, where $(p_0,q_0)$ is determined as follows:
\begin{enumerate}
\item If $P(x,y)$ is dominated by an edge, then $(p_0,q_0)$ is the left vertex of that edge.
In this case, let $-1/m_0$ be the slope of that edge. 
\item If $P(x,y)$ is dominated by a vertex, then $(p_0,q_0)$ is that vertex. In this case, let
 $-1/m_0$ and  $-1/m_{0,r}$ be the slopes of the edges on the left and right of  $(p_0,q_0)$. 
\end{enumerate}
For $(x,y)\in\mathcal R\in U_{0,g}(\sigma,\rho)$, one has $x\sim \sigma$ and $|y|\sim \rho\sigma^{m_0}$.  Notice 
$\sigma^{m_{0,r}-m_0}\lesssim \rho\lesssim 1$, where we set $m_{0,r}=m_0$ if $U_{0,g}$ is defined by an edge.  
Since $|x^{p_0}y^{q_0}| $ is a dominant term in $P(x,y)$ for $(x,y)\in U_{0,g}$, one has $|x^{p_0}y^{q_0}| \gtrsim |x^{n-3-\beta} y^\beta|$ and in particular 
\begin{align}\label{low35}
|P(x,y)| \sim |x^{p_0}y^{q_0}|\gtrsim |x^{n-3-\beta} y^\beta|\sim \sigma^{n-3-\beta} (\rho\sigma^{m_0})^{\beta}.
\end{align}
Notice $\Delta x \sim \sigma \gtrsim \Delta y \sim \rho \sigma^{m_0}$ and the cardinality of $\mathcal R\in U_{0,g}$ is 
about $\f \sigma {\Delta x}$, which is a uniformly bounded number. 
Setting $p =\frac  8 3$, $q = 2$, $r =\f 8 3$ in (\ref{OS01}) yields 
  \begin{align*}
\|\Lambda_{0,g}[\sigma,\rho]\|& \lesssim \f \sigma {\Delta x} \| \Lambda_\mathcal R \|
\\
&\lesssim 
|\lambda\sigma^{n-3-\beta} (\rho\sigma^{m_0}) ^{\beta}|^{-1/4}\Delta x ^{\frac 3 8}\Delta y^{\frac 1 2 } \Delta x ^{\frac 3 8}
 \\
 &= |\lambda\sigma^{(n-8)+(m_0-1)(\beta-2)} \rho ^{\beta-2}|^{-1/4}.
\end{align*}
On the other hand (\ref{SZ01}) yields 
\begin{align}
\|\Lambda_{0,g}[\sigma,\rho]\| \lesssim \f \sigma {\Delta x }\| \Lambda_\mathcal R \|& \lesssim \sigma \cdot \Delta y\sim \rho \sigma^{1+m_0}. 
\end{align}
For convenience in later arguments, let us 
define the quantities $S_1(\sigma,\rho)$, $S_2(\sigma,\rho)$ and $S(\sigma,\rho)$ by 
\begin{align}
S_1(\sigma,\rho) &  = |\lambda\sigma^{(n-8)+(m_0-1)(\beta-2)} \rho ^{\beta-2}|^{-1/4} , \nonumber \\
S_2 (\sigma,\rho) & = \rho \sigma^{1+m_0}, \nonumber \\
S (\sigma,\rho)\,\,\, & = (S_1(\sigma,\rho))^{\frac{4}{\theta+2}} (S_2(\sigma,\rho))^{\frac{\theta-2}{\theta + 2}} = 
 (|\lambda|^{- 1 }\sigma ^A\rho ^B)^{\frac{1}{\theta+2}}, \label{SS01}
\end{align}
where 
\begin{align}
\begin{cases}\label{AB01}
A=2\theta -(n-4)+(\theta-\beta)(m_0-1), 
\\
B= \theta -\beta.
\end{cases}
\end{align}
Notice that $\| \Lambda_{0,g}[\sigma,\rho] \| \lesssim S(\sigma,\rho)$ for any $\theta \geq 2$ and that $A ,B \geq 0$ 
given \eqref{theta0}.  
We postpone the summation over $\sigma$ and $\rho$ until Section \ref{sumsec}. 

\vspace{.2in}
\subsection{{Case 2}: $k\geq 1$ and $(m_0,r_0)\neq (1, -1)$.} 

Then in the good region $U_{k,g}$,  $|DS (x,y)| \sim |x_k^{p_k} y_k^{q_k}|$, where $y_k$ is given by
\begin{align*}
\begin{cases}
y =\gamma_k(x)+y_kx^{m_0+\dots+m_{k-1}},
\\
 x^{m_{k,r}} \lesssim |y_k|\lesssim x^{m_k}.
\end{cases}
\end{align*}
Here $\gamma_k(x)$ is given by (\ref{gamma1x}) or (\ref{gamma2x}), and 
$m_k =m_{k,r}$ if the good region $U_{k,g}$ is defined by an edge with slope $-\f 1 {m_k}$; if not, then $m_{k,r}>m_k$ and $U_{k,g}$ is defined by a vertex $V_k$ with $-\f 1 {m_k}$ and $-\f 1 {m_{k,r}} $ being the slopes
of the edges left and right of  $V_k$, respectively. 
For $(x,y)\in\mathcal R\in U_{k,g}(\sigma,\rho)$ one has 
$x\sim \sigma $ and $|y_k|\sim \rho \sigma^{m_k}$,
where  $\sigma^{m_{k,r} - m_k }\lesssim \rho \lesssim 1$. 
Since the cardinality of $\mathcal R\in U_{k,g}(\sigma,\rho)$ is bounded by a constant times $\sigma /\Delta x$, 
employing (\ref{OS01}) yields 
\begin{align}
\|\Lambda_{k,g}[\sigma,\rho]\|\lesssim \f \sigma {\Delta x} \|\Lambda_\mathcal R\| 
\lesssim  \f \sigma {\Delta x}   |\lambda \sigma^{p_k+m_kq_k}\rho^{q_k}|^{-1/4}   |\Delta x|^{\frac 3 8}   |\Delta y|^{\frac 1 2}   |\Delta x+\Delta y|^{\frac 3 8}.
\end{align}
By (\ref{ind1}) and (\ref{IT5}), one has 
\begin{align*}
|P (x,y)| \sim |x_k^{p_k} y_k^{q_k}|\sim |\sigma^{p_k+m_kq_k} \rho^{q_k}|\gtrsim |\sigma^{p_0+m_0q_0 +s_0\sum_{1\leq j \leq k}  m_j} \rho^{s_{k-1}}|.
\end{align*}
Notice also that
\begin{align}
\begin{cases}
p_0+m_0q_0 \leq (n-3-\beta) +m_0 \beta, 
\\
s_ 0\geq s_{j-1} \,\,\,\mbox{for}\,\, 1\leq j \leq k,
\\
\Delta x \sim   \rho \sigma^{m_0+\dots +m_{k}}  \cdot \sigma ^{1-m_0} \sim \Delta y \cdot \sigma ^{1-m_0}\geq \Delta y .
\end{cases}
\end{align}
Therefore
\begin{align}
\|\Lambda_{k,g}[\sigma,\rho]\|&
\lesssim   |\lambda \sigma^{\big((n-3-\beta) + \beta m_0\big)+s_0 (m_1+\dots +m_{k})}\rho^{s_0}|^{-\frac 1 4} \sigma \Delta y ^{\frac 1 4}\left (\frac{\Delta y}{ \Delta x}\right)^{\frac 1 4} \nonumber
\\
&\sim  |\lambda \sigma^{(n-8 )  + (\beta-2) (m_0-1)+(s_0 -1)(m_1+\dots +m_{k})}\rho^{s_0-1}|^{-\frac 1 4}. \label{news1}
\end{align}
As in the previous case, we define $S_1(\sigma,\rho)$ to be the quantity on the right-hand side of \eqref{news1}. The estimate (\ref{SZ01}) gives 
\begin{align*}
\|\Lambda_{k,g}[\sigma,\rho]\|&\lesssim \frac \sigma {\Delta x }   \Delta x\cdot \Delta y \sim \rho \sigma^{1+m_0+\dots +m_{k}}.
\end{align*}
In this case we define $S_2(\sigma,\rho)$ to equal $\rho \sigma^{1+m_0+\dots +m_{k}}$, and keeping the definition of $S(\sigma,\rho)$ as the same convex combination of $S_1(\sigma,\rho)$ and $S_2(\sigma,\rho)$ that was used in \eqref{SS01}, it follows that
\begin{align}\label{SS02}
S^{\theta+2}(\sigma,\rho) =S_1(\sigma,\rho)^4S_2(\sigma,\rho)^{ \theta -2}=|\lambda|^{-1}\sigma^A\rho^B,
\end{align}
where
\begin{align}\label{AB02}
\begin{cases}
A=(m_0-1)(\theta-\beta )+(\theta-1-s_0)(m_1+\dots +m_{k-1})+(2\theta-n+4),
\\
B= \theta -1-s_0.
\end{cases}
\end{align}
Notice that $s_0$ is the vanishing order of the edge of $\mathcal N(P)$ with slope $-1/m_0$. By our assumption $s_0\leq \max\{d_0, d_1\}\leq \frac n 2 -3\leq \theta -1$ and thus $A, B\geq 0$.  
\vspace{.2in}

\subsection{Case 3: $k= 1$ and $(m_0,r_0)=(1,-1)$.} In the good region $U_{1,g}$, $|P(x,y)|\sim |x_1^{p_1}y_1^{q_1}|$. 
Then $q_1\leq s_0 =\gamma \leq \kappa$ and $p_1 = p_0+m_0 q_0 = p_0+q_0 =n-3$. 
For $(x,y)\in U_{1,g}(\rho,\sigma)$, $x\sim \sigma $ and $|x+y|\sim \rho\sigma$. Then $U_{n,g}(\sigma,\rho)$ is the union of two parallelograms
with sides parallel to the $y$-axis and to the line $x+y=0$.  
One of them, denoted by $U_{n,g}(\sigma,\rho,+)$, lies above the line $x+y =0$, and the other, denoted by $U_{n,g}(\sigma,\rho,-)$, is below $x+y =0$. 
The treatment of these two cases is essentially the same, so we focus on $U_{n,g}(\sigma,\rho,+)$. 
Unlike the previous cases, we must consider $U_{n,g}(\sigma,\rho,+)$ as a whole rather than divide it into axis-parallel rectangular boxes. In other words, we consider the amplitude
$$
\phi_{n,g}[\sigma,\rho](x,y) =\sum_{\mathcal R\in U_{n,g}(\sigma,\rho,+)}\phi_\mathcal R(x,y).
$$
The support of $\phi_{n,g}[\sigma,\rho]$ is contained in the region $U_{n,g}^*(\sigma,\rho,+)$, an enlarged version of $U_{n,g}(\sigma,\rho,+)$, with $x$-cross section comparable to $\Delta x$ and $y$-cross section comparable to $\Delta y$, where $\Delta x =\Delta y = \rho \sigma$. Moreover, 
$$
|\p_y^{b}\phi_{n,g}[\sigma,\rho]|\lesssim \Delta y^{-b} \q{\rm for}\q b=0, 1, 2 .
 $$ 
 Notice that $U_{n,g}^*(\sigma,\rho,+)$ is convex in the $\vec u$, $\vec v$ and $\vec w$ directions and that $P(x,y)$ is still comparable to a monomial in $U_{n,g}^*(\sigma,\rho,+)$. 
 Thus we can apply the $(\frac 8 3, \f 8 3, 2)$ estimate from Theorem \ref{local}, which gives  
\begin{align*}
\|\Lambda_{k,g}[\sigma,\rho]\| &\lesssim |\lambda \sigma^{n-3}\rho^{\gamma}|^{-\frac  1 4} \sigma^{\frac 3 8}\sigma^{\frac 3 8} (\rho\sigma)^{\frac 1 2}
\\
&= |\lambda \sigma^{n-8}  \rho^{\gamma -2}|^{-\frac 1 4}. 
\end{align*}
We also have  
\begin{align}
\|\Lambda_{k,g}[\sigma,\rho]\|  \lesssim \sigma (\rho \sigma). 
\end{align}
As in the previous cases, we define $S_1(\sigma,\rho)$ to be $|\lambda \sigma^{n-8}  \rho^{\gamma -2}|^{-\frac 1 4}$, $S_2(\sigma,\rho)$ to be $\sigma (\rho \sigma)$, and $S(\sigma,\rho)$ to be
\begin{align}\label{SS03}
S^{\theta+2}(\sigma,\rho)=S_1(\sigma,\rho)^4S_2(\sigma,\rho)^{\theta- 2 }=|\lambda|^{-1}\sigma^A\rho^B,
\end{align}
where 
\begin{align}\label{AB03}
\begin{cases}
A=2\theta -n+4,
\\
B= \theta-\gamma,
\end{cases}
\end{align}
which are both non-negative.

\subsection{Case 4: $k\geq 2$ and $(m_0,r_0)=(1,-1)$. }
The case for $k\geq 2$ is similar to the previous one; this time we divide $U_{k,g}(\sigma,\rho)$ into boxes with sides parallel to the $y$-axis and the line $x+y=0$. 
In the good region $U_{k,g}$, $|P(x,y)|\sim |x_k^{p_k} y_k^{q_k}|$, where    
the change of variables is 
\begin{align}
y&=\gamma_k(x) +y_kx^{m_0+\dots +m_{k-1}}.
\end{align}
Write 
$$
\gamma_k(x) = r_0x^{m_0}+r_1x^{m_0+m_1}+\dots+r_{k-1}x^{m_0+\dots +m_{k-1}}+\xi(x),
$$
where $\xi(x)$ is the sum of the remaining terms (if any) of higher degree in $x$. 
Plug in $(m_0, r_0)=(1,-1)$ and  let 
\begin{align}
z=y+x=r_1x^{m_0+m_1}+\dots+r_{k-1}x^{m_0+\dots +m_{k-1}}+\xi(x)+y_k x^{m_0+\dots +m_{k-1}}.
\end{align}
Cover $U_{k,g}(\sigma,\rho)$ by $\Delta x\times \Delta z$ parallelograms whose sides are parallel to the $y$-axis and the line $z=0$, where 
\begin{align}
\begin{cases}
\Delta z\sim \rho \sigma^{m_0+\dots +m_{k}},
\\
\Delta x \sim \frac {\Delta z} {\frac {dz}{dx}} = \rho \sigma^{m_0+\dots +m_{k}} \sigma^{-m_1}.
\end{cases}
\end{align}
The number of such parallelograms is about $\frac \sigma {\Delta x}$.  
Consider the corresponding smooth partition and
employ the $(\frac 8 3, \f 8 3, 2)$ estimate in Theorem \ref{local}; we conclude that  
\begin{align*}
\|\Lambda_{k,g}[\sigma,\rho]\| &\lesssim |\lambda \sigma^{p_k+m_kq_k}\rho^{q_k}|^{-\frac  1 4}  \frac \sigma {\Delta x} \Delta x^{\frac 3 8}\Delta x ^{\frac 3 8} \Delta z ^{\frac 1 2}
\\
&\lesssim (\lambda\sigma^{(n-3)+s_0m_1+s_1(m_2+\dots+m_k)-4-m_1-(m_0+\dots+m_k)}\rho^{s_1-1})^{-\f 14}
\\
&= (\lambda\sigma^{(n-8)+(s_0-2)m_1+(s_1-1)(m_2+\dots+m_k)}\rho^{s_1-1})^{-\f 14},
\end{align*}
where, as in previous cases, we define $S_1(\sigma,\rho)$ to be the final quantity on the right-hand side.
In the above estimates, we used the fact that for $k\geq 2$, 
$
s_0\geq q_2\geq  q_k 
$
and
\begin{align*}
p_k+m_kq_k &\leq p_0+m_0q_0+m_1q_{1,l}+\dots +m_{k}q_{k,l}
\\
&\leq (n-3)+s_0 m_1 +s_1(m_2+\dots +m_{k}).
\end{align*}
Notice also 
\begin{align}
\|\Lambda_{k,g}[\sigma,\rho]\|   \lesssim \sigma \rho \sigma^{m_0+\dots +m_{k}};  
\end{align} 
in this case, we define $S_2(\sigma,\rho)$ to equal $\sigma \rho \sigma^{m_0+\dots +m_{k}}$.
Finally, as in all the previous cases, we define $S(\sigma,\rho)$ so that
\begin{align}\label{SS04}
S^{\theta+2}(\sigma,\rho)=S_1(\sigma,\rho)^4S_2(\sigma,\rho)^{\theta -2}=|\lambda|^{-1}\sigma^A\rho^B,
\end{align}
where 
\begin{align}\label{AB04}
\begin{cases}
A =(\theta-s_0)m_1+(\theta -1 -s_1)(m_2+\dots+m_{k-1}) +(2\theta -n+4),
\\
B= \theta -1-s_1.
\end{cases}
\end{align}
Both $A$ and $B$ are non-negative since $s_0 = \gamma\leq \theta$ and $s_1\leq d_1 \leq \theta -1$.

\subsection{Summing over $\sigma$ and $\rho$} 
We focus on $\theta >2$ first, which holds automatically when $n\geq 9$. \label{sumsec}
In this case, $S(\sigma,\rho)$ is a true convex combination of $S_1(\sigma,\rho)$ and $S_2(\sigma,\rho)$; see
(\ref{SS01}), (\ref{SS02}), (\ref{SS03}) and (\ref{SS04}).
It is not difficult to see that if $A>0$ or $B>0$ for all possible definitions (\ref{AB01}), (\ref{AB02}), (\ref{AB03}) and (\ref{AB04}), then the sum of 
$\|\Lambda_{k,g}[\sigma,\rho]\|$ 
over $\sigma$ and $\rho$ is bounded above by a multiple of $|\lambda|^{-\f 1 {\theta+2}}$.  Indeed if $A>0$, 
fix $\sigma$, then  
$$
 \sum_{\rho >0} \min\{S_1(\sigma,\rho), S_2(\sigma,\rho)  \}  \lesssim  S(\sigma,\rho) \lesssim  \sigma^{\frac{4A}{\theta+2}}\lambda^{-\frac 1 {\theta +2}},
$$
which is absolutely summable over $\sigma>0$. Thus 
\begin{align}\label{SUM09}
\sum_{\sigma,\rho} \|\Lambda_{k,g}[\sigma,\rho]\| \lesssim |\lambda|^{-\frac 1 {\theta +2}}. 
\end{align}
The case when $B>0$ is similar, but we need to use Fubini's Theorem to switch the order of $\sigma$ and $\rho$ in the sum first.
One can see that $\kappa >\f n 2- 2$ implies $A>0$ and $\kappa< \f n 2 -2$ implies $B>0$ for all cases. 
Therefore, for $\theta >2$ and $\kappa \neq \f n 2 -2$   
\begin{align}
\|\Lambda_{k,g}\| \lesssim |\lambda|^{-\f 1 {\max\{\kappa+2, n /2\}}}. 
\end{align}
In what follows, we assume $\kappa =\f n 2 -2 $ and $\theta >2$.  This implies $\theta =\kappa  =\f n 2 -2$. 
Notice that
$\beta =\kappa$ implies $A=B=0$ in  
(\ref{AB01}) and $\gamma =\kappa$ implies $A=B=0$  (\ref{AB03}). 
If $m_0=1$ and $s_0 =d_0 = \f n  2 -3$, then $A=B=0$ in (\ref{AB02}).  
In these three cases, we encounter an extra factor of $\log(2+|\lambda|)$ 
when summing over $\sigma$ and $\rho$.  To see this, first observe that when $\rho <|\lambda|^{-\frac 1 {\theta +2}}$, we can use $S_2(\sigma,\rho)$ to 
control $\|\Lambda_{k,g}[\sigma,\rho]\|$, which gives (\ref{SUM09}) when the sum is restricted of all $\sigma$ and all such $\rho$. The cardinality of 
dyadic numbers 
$ |\lambda|^{-\frac 1 {\theta +2}} \leq \rho\lesssim 1$ is bounded by a constant times $\log (2+|\lambda|)$ and we have 
$\sum_{\sigma} \|\Lambda_{k,g}[\sigma,\rho]\| \lesssim |\lambda|^{-\frac 1 {\theta +2}} $ for each fixed $\rho$. Thus we incur only an extra factor of $\log (2+|\lambda|)$ over the right-hand side of \eqref{SUM09} itself.
The same thing happens in the north or south regions $|y|\gtrsim |x|$, when $\alpha =\kappa =\f n 2 -2$.  
But if $\alpha, \beta, \gamma, d_0+1 <\f n 2 -2$ and $d_1+1 =\f n 2-2$, the $\log(2+|\lambda|)$ loss is not necessary.  
Indeed, $B>0$ in (\ref{AB01}) and (\ref{AB03}).  If $B=0$ in (\ref{AB02}), i.e. $s_0 = \theta -1 =d_1$, the definition of $d_1$ implies $m_0\neq 1$ and hence $m_0>1$. Since $\theta>\beta$ thus $A>0$. In (\ref{AB04}), $\theta -s_0 =\theta -\gamma >0$ and $m_1 \neq 0$, thus $A>0$. 

Thus, we have verified the all the estimates in Theorem \ref{main3} (1)
and the estimates in (2) when $\theta >2$ (i.e., $n\geq 9$ or $5\leq n\leq 8$ and $\kappa >2$).

To verify the remaining cases of (2) in Theorem \ref{main3},
assume $\theta \leq 2$, i.e., $n \leq 8$ and $\kappa \leq 2$. 
Notice that the exponents of $\sigma$ and $\rho$ in $S_1(\sigma,\rho)$ are always non-negative. Therefore the following holds 
\begin{align}\label{SUM10}
\sum_{\sigma, \rho }\|\Lambda_{k,g}[\sigma,\rho]\| \lesssim  |\lambda|^{-\frac 1 {4}} |\log (2+|\lambda|)|^2. 
\end{align}
In some cases, the exponents of $\log (2+|\lambda|)$ can be decreased to 1 or even 0. For example, if $n \leq 7$, then the
exponent of $\sigma$ in $S_1(\sigma,\rho)$ is positive and one can remove one $\log(2+|\lambda|)$ in (\ref{SUM10}). 
 If in addition $\kappa=1$, then the exponent of $\rho$ in
$S_1(\sigma,\rho)$ is also positive and (\ref{SUM10}) holds without the $|\log (2+|\lambda|)|^2$ term. But
these estimates are less interesting for they are not sharp in general.

\subsection{Sharpness}
First of all, the $|\lambda|^{-\f 2 n}$ decay in Theorem \ref{main3} is the highest possible that one can achieve for real analytic phases. In particular,
the results in the first part (with $\mu =0$) of
Theorem \ref{main3} (1) are sharp. 
One way to see this is via extrapolation. 
Indeed, $P_{n-3}\neq 0$ implies the multiplicity of the quotient of $S$ by the class of functions annihilated by $D$ is $n$.
Extrapolating the following sharp estimate in \cite{X2013}  
$$
|\Lambda (f,g,h)| \lesssim 2^{-\f1 {2n}}\|f\|_2\|g\|_2\|h\|_2.
$$
with the trivial $L^{3/2}\times L^{3/2}\times L^{3/2}$ non-decay estimate, we see that 
the highest possible decay for $L^{\infty}\times L^{\infty}\times L^{\infty}$ is $|\lambda|^{-\f 2 n}$. 
What is amazing here is that this highest possible decay can only be achieved when $f$, $g$ and $h$ 
all belong to $L^\infty$.

Another way to see that the $|\lambda|^{-\f 2 n}$ decay is the highest possible is via sublevel set estimates. 
Recall that for $\delta \in (0,1)$, a $|\lambda| ^{-\delta}$ decay estimate of \ $L^{\infty}\times L^{\infty}\times L^{\infty}$ 
implies the following uniform sublevel set estimate  
$$
|\{(x,y)\in \supp \phi:|S(x,y)-P(x)-Q(y)-R(x+y)| <\epsilon \}|\leq C \epsilon^{-\delta}, 
$$
for some positive constant $C$ uniform of all measurable functions $P, Q$ and $R$.  
We can always write $S$ as the sum of homogeneous polynomials 
$$
S(x,y) =\sum_{k =n}^\infty S_k(x,y),
$$  
since the term $(S_0(x,y)+\dots+S_{n-1}(x,y))$ is assumed to be annihilated by $D$ and thus is expressible as 
a sum of functions of $x$,  $y$ and $(x+y)$.  
There is a positive number $K$ (depending on $S$) such that for all $\epsilon >0$ sufficiently small,  
$$
|S(x,y)|<\epsilon \q {\rm given }\q |x|, |y| <\epsilon ^{\f 1 n }/K.
$$
Consequently
$$
|\{(x,y)\in \supp \phi:  |S(x,y)|<\epsilon \}| \geq \epsilon ^{\f 2 n }/K^2.  
$$
Therefore any decay of the $L^{\infty}\times L^{\infty}\times L^{\infty}$ estimate can not be better than $|\lambda|^{-\f 2 n}$.

When $\kappa = \f n 2 -2$ ($n$ is even) and $n\geq 9$, the extra $\log (2+|\lambda|)$ term is indeed necessary if at least one of $\alpha$, $\beta$, $\gamma$, $d_0+1$ is equal to $\f n 2 -2$. 
In particular, the estimates in the second part (with $\mu=1$) of Theorem \ref{main3} (1) are also sharp. To prove this, we use the following facts: if $(\f n 2, \f n2)$ is a vertex (and thus the main face) of $\mathcal N(S)$, then 
\begin{align}\label{SUB}
|\{(x,y)\in \supp \phi:|S(x,y)| <\epsilon \}|\sim \epsilon^{\f 2 n} \log (2+|\lambda|). 
\end{align}
This estimate is also true if there is a local coordinate change with the Hessian bounded below by a non-zero constant 
such that the main face 
of $\mathcal N(S)$ under this new coordinate is $(\f n 2, \f n2)$. 
Now if $\alpha=\kappa = \f n 2 -2$, one can find a real analytic function $S$ 
such that $x$ is a factor of $S_n(x,y)$ of order $\f n 2$. 
Then $(\f n 2, \f n 2)$ is a vertex of $\mathcal N(S)$. 
In particular, (\ref{SUB}) holds with this function. 
Then the $L^{\infty}\times L^{\infty}\times L^{\infty}$ estimate with decay $|\lambda| ^{-\f 2 n}\log (2+|\lambda| )$ has to be sharp 
for it implies the corresponding uniform sublevel set estimate. 
The case for $\beta = \kappa = \f n 2 -2$ is similar. 
For $\gamma =\kappa = \f n 2 -2$, one can also find an $S$ such that $(x+y)$ is a factor of $S_{n}(x,y)$ of order $\f n 2$.  
Under the new coordinate $(x+y, y)$, $(\f n 2, \f n 2)$ is a vertex of $S$. 
When $d_0 +1 =\kappa = \f n 2 -2$, 
there is a linear factor $(x+by)$, other than $x$, $y$ and $(x+y)$, whose order in $DS_n(x,y)$ is $\f n 2 -3$.
Consequently, there is a $S$ such that the order of $(x+by)$ in $S_n(x,y)$ is exactly $\f n 2$. Under the new coordinate $(x, x+by)$, $S$ also has $(\f n 2, \f n 2 )$ as its vertex.

\bibliographystyle{alpha}

\end{document}